\documentclass[a4paper,11pt]{article}
\usepackage[utf8]{inputenc}

\usepackage{amsthm}
\usepackage{amssymb}
\usepackage{amsmath}
\usepackage{mathtools}
\usepackage[mathscr]{euscript}
\usepackage[normalem]{ulem}

\usepackage[nocompress]{cite} % order citations numerically

\usepackage[usenames,dvipsnames]{xcolor}
\usepackage[xcolor]{changebar}
\usepackage{pgf, tikz, pgfplots}
\pgfplotsset{compat=1.17}
\usepackage{appendix}
\usetikzlibrary{calc, arrows, shapes, positioning, patterns, angles, quotes, intersections, through, backgrounds}
\usetikzlibrary{decorations.markings}
\usepackage[titles]{tocloft}
\usepackage{fancyhdr}
\usepackage{orcidlink}
\usepackage{hyperref}
\hypersetup{    
}
\usepackage[inline]{enumitem}
\newcommand{\enumlabelformat}{\roman}
\newcommand{\enumlabelfont}[1]{#1}
\newlength{\thelabelsep}
\setlist{labelsep=\thelabelsep}
\setlist[enumerate,1]{font=\enumlabelfont,label=(\enumlabelformat*),leftmargin=2.5em}
\setlist[itemize]{leftmargin=2.5em,label=$-$}

\newcounter{inlineenum}
\renewcommand{\theinlineenum}{\enumlabelformat{inlineenum}}
\usepackage{multirow}

\numberwithin{equation}{section}

\theoremstyle{definition}
\newtheorem{dfn}{Definition}[section]

\newtheorem{thm}[dfn]{Theorem}
\newtheorem{pop}[dfn]{Proposition}
\newtheorem{lem}[dfn]{Lemma}
\newtheorem{cor}[dfn]{Corollary}

\newtheorem{manualtheoreminner}{Theorem}
\newenvironment{introtheorem}[1]{% 
\renewcommand\themanualtheoreminner{#1}%
\manualtheoreminner
}{\endmanualtheoreminner}

\theoremstyle{remark}

\newtheorem{ex}[dfn]{Example}

\AtBeginEnvironment{ex}{%
\pushQED{\qed}%
}
\AtEndEnvironment{ex}{\popQED\endex}

\AtBeginEnvironment{rem}{%
\pushQED{\qed}%
}
\AtEndEnvironment{rem}{\popQED\endex}

\newenvironment{acknowledgements}{%
\begin{abstract}
}{%
\end{abstract}
}

\fancypagestyle{firstpage}
{
\fancyhead[L]{}    
\fancyhead[R]{DMUS-MP-23/14}
}

\DeclareMathOperator{\diam}{diam}
\newcommand{\diamfin}{\ensuremath{\diam_{\mathrm{fin}}}}

\let\epsilon\varepsilon 
\let\phi\varphi

\newcommand{\R}{\mathbb{R}}

\newcommand{\LLS}{Lo\-rentz\-ian length space }
\newcommand{\LLSn}{Lo\-rentz\-ian length space}
\newcommand{\LLSs}{Lo\-rentz\-ian length spaces }

\newcommand{\LpLS}{Lo\-rentz\-ian pre-length space }
\newcommand{\LpLSn}{Lo\-rentz\-ian pre-length space}
\newcommand{\LpLSs}{Lo\-rentz\-ian pre-length spaces }
\newcommand{\LpLSsn}{Lo\-rentz\-ian pre-length spaces}

\newcommand*{\bx}{\bar{x}}
\newcommand*{\by}{\bar{y}}
\newcommand*{\bz}{\bar{z}}

\newcommand*{\bp}{\bar{p}}
\newcommand*{\bq}{\bar{q}}
\newcommand*{\br}{\bar{r}}

\newcommand{\tp}{\tilde{p}}
\newcommand{\tq}{\tilde{q}}
\newcommand{\trr}{\tilde{r}}

\newcommand{\tx}{\tilde{x}}

\makeatletter
\let\save@mathaccent\mathaccent
\newcommand*\if@single[3]{%
\setbox0\hbox{${\mathaccent"0362{#1}}^H$}%
\setbox2\hbox{${\mathaccent"0362{\kern0pt#1}}^H$}%
\ifdim\ht0=\ht2 #3\else #2\fi
}
\newcommand*\rel@kern[1]{\kern#1\dimexpr\macc@kerna}
\newcommand*\widebar[1]{\@ifnextchar^{{\wide@bar{#1}{0}}}{\wide@bar{#1}{1}}}
\newcommand*\wide@bar[2]{\if@single{#1}{\wide@bar@{#1}{#2}{1}}{\wide@bar@{#1}{#2}{2}}}
\newcommand*\wide@bar@[3]{%
\begingroup
\def\mathaccent##1##2{%
\let\mathaccent\save@mathaccent
\if#32 \let\macc@nucleus\first@char \fi
\setbox\z@\hbox{$\macc@style{\macc@nucleus}_{}$}%
\setbox\tw@\hbox{$\macc@style{\macc@nucleus}{}_{}$}%
\dimen@\wd\tw@
\advance\dimen@-\wd\z@
\divide\dimen@ 3
\@tempdima\wd\tw@
\advance\@tempdima-\scriptspace
\divide\@tempdima 10
\advance\dimen@-\@tempdima
\ifdim\dimen@>\z@ \dimen@0pt\fi
\rel@kern{0.6}\kern-\dimen@
\if#31
\overline{\rel@kern{-0.6}\kern\dimen@\macc@nucleus\rel@kern{0.4}\kern\dimen@}%
\advance\dimen@0.4\dimexpr\macc@kerna
\let\final@kern#2%
\ifdim\dimen@<\z@ \let\final@kern1\fi
\if\final@kern1 \kern-\dimen@\fi
\else
\overline{\rel@kern{-0.6}\kern\dimen@#1}%
\fi
}%
\macc@depth\@ne
\let\math@bgroup\@empty \let\math@egroup\macc@set@skewchar
\mathsurround\z@ \frozen@everymath{\mathgroup\macc@group\relax}%
\macc@set@skewchar\relax
\let\mathaccentV\macc@nested@a
\if#31
\macc@nested@a\relax111{#1}%
\else
\def\gobble@till@marker##1\endmarker{}%
\futurelet\first@char\gobble@till@marker#1\endmarker
\ifcat\noexpand\first@char A\else
\def\first@char{}%
\fi
\macc@nested@a\relax111{\first@char}%
\fi
\endgroup
}
\makeatother

\newcommand{\lm}[1]{\mathbb{L}^2(#1)}
\newcommand{\ma}{\ensuremath{\measuredangle}}

\newcommand{\mb}[1]{\mathbb{#1}}

\providecommand\given{} 
\newcommand\SetSymbol[1][]{
\nonscript\,#1\vert \allowbreak \nonscript\,\mathopen{}}
\DeclarePairedDelimiterX\Set[1]{\lbrace}{\rbrace}%
{ \renewcommand\given{\SetSymbol[\delimsize]} #1 }

\newcommand{\rr}{\mathbb{R}}

\title{ \vspace{-3.5em} A Toponogov globalisation result for Lorentzian length spaces \\[1ex]
{\normalsize \textit{Dedicated to the memory of Stephanie Alexander}}}

\author{Tobias Beran,\footnote{\href{mailto:tobias.beran@univie.ac.at}{tobias.beran@univie.ac.at}, Department of Mathematics, University of Vienna, Oskar-Morgenstern-Platz 1, 1090 Wien, Austria. \,\orcidlink{0000-0002-2813-0099}} \:
John Harvey,\footnote{\href{mailto:HarveyJ13@cardiff.ac.uk}{harveyj13@cardiff.ac.uk}, School of Mathematics, Cardiff University, Senghenydd Road, Cardiff, CF24 4AG, UK. \,\orcidlink{0000-0001-9211-0060}} \:
Lewis Napper,\footnote{\href{mailto:lewis.napper@surrey.ac.uk}{lewis.napper@surrey.ac.uk}, Department of Mathematics, University of Surrey, Stag Hill Campus, Guildford, GU2 7XH, UK. \,\orcidlink{0009-0001-6961-7055} } \:
Felix Rott\footnote{\href{mailto:frott@sissa.it}{frott@sissa.it}, SISSA, Via Bonomea 265, 34136 Trieste, Italy. \,\orcidlink{0000-0001-7314-0889}} \\ %\footnotemark[\value{footnote}],
}

\begin{document}

\date{\today}

\maketitle
\thispagestyle{firstpage}
\vspace{-1.5em}

\begin{abstract}
In the synthetic geometric setting introduced by Kunzinger and S{\"a}mann, we present an analogue of Toponogov's Globalisation Theorem which applies to Lorentzian length spaces with lower (timelike) curvature bounds. Our approach utilises a ``cat's cradle'' construction akin to that which appears in several proofs in the metric setting. On the road to our main result, we also provide a lemma regarding the subdivision of triangles in spaces with a local lower curvature bound and a synthetic Lorentzian version of the Lebesgue Number Lemma. Several properties of time functions and the null distance on globally hyperbolic Lorentzian length spaces are also highlighted. We conclude by presenting several applications of our results, including versions of the Bonnet--Myers Theorem and the Splitting Theorem for Lorentzian length spaces with local lower curvature bounds, as well as discussion of stability of curvature bounds under Gromov--Hausdorff convergence.
\end{abstract}
\vspace{1em}

\begin{acknowledgements}
We want to thank James Grant, Stacey Harris and Didier Solis for valuable input in the early stages of this work. We also want to acknowledge the kind hospitality of the Erwin Schrödinger International Institute for Mathematics and Physics (ESI), where this project was initiated. 

This work was supported by research grant P33594 of the Austrian Science Fund FWF and by a UKRI Future Leaders Fellowship [grant number MR/W01176X/1].
This material is based in part upon work supported by the National Science Foundation under Grant No. DMS-1928930, while JH was in residence at the Simons Laufer Mathematical Sciences Institute in Berkeley, California, during Fall 2024.
\end{acknowledgements}

\newpage
\hypersetup{linkcolor=black}
\tableofcontents
\vfill\noindent
\medskip

\hypersetup{linkcolor={red!90!black}}
\newpage

\section{Introduction}\label{sec:intro}

Recall that a metric space $(X,d)$ is called a length (or intrinsic) space if its distance function $d(p,q)$ can be recovered as the infimum of the length of curves joining $p$ to $q$. This is the realm of so-called synthetic geometry, which can be seen as a generalisation of Riemannian geometry to spaces of lower regularity. Such spaces have proven an essential tool in the study of geometric flows \cite{AGS08,GN21}, optimal transport \cite{Stu06,Vil09}, and bounds on the number of finite subgroups of fundamental groups, see \cite[Corollary 9.3.2]{BBI01}. One notion that frequently arises in this setting is the concept of curvature bounds; a metric length space is said to have a lower (or upper) curvature bound if a given comparison condition\footnote{Several of these comparison conditions are on display in \cite[Theorem 8.30]{AKP19} and are shown to be equivalent for complete metric length spaces.} holds on a neighbourhood of each point $x\in X$, \cite{AKP19, BBI01}. These conditions are used to tame some of the more erratic behaviours of metric length spaces, so that they act more like their Riemannian counterparts, while not requiring smoothness.
\medskip

A vast amount of theory has been developed concerning spaces which exhibit global curvature bounds (where the comparison condition holds on the whole space) and their properties \cite{Gro78, Gro99}. The preservation of curvature bounds along sequences of spaces which converge in the Gromov--Hausdorff topology is a prime example \cite{BGP92,Kap02}.  As such, it is pertinent to ask when a space with a known local curvature bound also possesses a global one, that is, when does a curvature bound globalise? 

In the case of lower curvature bounds, this was first proven in two dimensions by Pizzetti \cite{Piz07} (see the history \cite{PZ11} for more details) and later independently re-proven by Alexandrov \cite{Ale51, Ale57}. These ``Toponogov Globalisation Theorems'' were popularised by Toponogov's proof for Riemannian manifolds in the late 1950s \cite{Top57, Top58, Top59}. Since then, Burago, Gromov and Perelman \cite{BGP92} and Plaut \cite{Pla91,Pla96} have extended the result to arbitrary complete metric length spaces, with refinements to their proofs being made by Alexander, Kapovitch, and Petrunin \cite{AKP19}, as well as Lang and Schroeder \cite{LS13}. A further generalisation regarding (not necessarily complete) geodesic spaces was also provided by Petrunin in \cite{Pet16}.
\medskip

Analogously to metric length spaces, in \cite{KS18}, Kunzinger and S{\"a}mann introduced the notion of a Lorentzian (pre-)length space, which facilitates the study of non-smooth Lorentzian geometry, with key applications in the investigation of spacetimes with low regularity metrics \cite{CG12, GKS19, GKSS20}, cones \cite{AGKS19}, and robust concepts of Gromov--Hausdorff convergence in the Lorentzian setting \cite{KS21, Mul23, MS23}. This synthetic Lorentzian picture also admits bounds on the so-called timelike curvature of a Lorentzian length space, via comparison conditions \cite{KS18,BMS22,BS22}. These timelike curvature bounds have been shown to behave like their metric counterparts in many circumstances and have hence been crucial for deriving Lorentzian equivalents to many metric results, including the Reshetnyak Gluing Theorem \cite{BR22, Rot22}, Splitting Theorem \cite{BORS23}, and a Bonnet--Myers style theorem for spaces with global lower timelike curvature bounds \cite{BNR23}. 
Timelike curvature bounds of Ricci type have also been defined and been shown to have good local-to-global properties \cite{Bra23, CM24}. 
\medskip

As for metric spaces, it is again pertinent to ask when a Lorentzian space with a local timelike curvature bound has a global one. In the smooth Lorentzian setting, the first result in this direction was achieved by Harris in \cite{Har82}, where a global comparison condition was inferred from lower timelike (sectional) curvature bounds. An Alexandrov's Patchwork approach was used by three of the present authors to answer this question for Lorentzian length spaces in the case of upper timelike curvature bounds in \cite{BNR23}, where globalisation results in the metric and Lorentzian settings were compared in detail. This paper is a continuation of that work and presents a solution for spaces with lower timelike curvature bounds, as well as several consequences of interest.
\medskip

The paper is organised as follows. We begin in Section \ref{sec:preliminaries} with a brief review of some basic yet crucial properties of Lorentzian (pre-)length spaces. We provide an overview of hyperbolic angles and how they may be used to describe curvature bounds, before discussing existence conditions for time functions and null distances with advantageous properties. The principal part of this paper is contained in Section \ref{sec:toponogov}, which begins with a series of supplementary results, including a Lorentzian analogue of the Lebesgue Number Lemma and a result concerning the splitting of triangles in Lorentzian length spaces with lower curvature bounds, in the spirit of the Gluing Lemma \cite{BR22}. We then proceed with a construction derived from the ``cat's cradle'' of Lang and Schroeder \cite{LS13} in the metric setting, with our main result being stated as follows:

\begin{introtheorem}{\ref{thm: LorentzianToponogov}}
Let $X$ be a connected, globally hyperbolic, regular Lorentzian length space with a time function $T$ and curvature bounded below by $K\in \mathbb{R}$ in the sense of angle comparison. 
Then each of the properties in Definition \ref{def-cb-ang} hold globally; in particular, the entire space $X$ is a $(\geq K)$-comparison neighbourhood and hence has curvature globally bounded below by $K$. 
\end{introtheorem}

This result globalises the notion of lower curvature bounds defined via ``angle comparison,'' as in Definition \ref{def-cb-ang}. Analogously to the metric setting, curvature bounds may also be characterised with respect to other comparison conditions, which can be shown to be equivalent (see \cite[Theorem 5.1]{BKR23} for a complete list of equivalent characterisations). Therefore, these conditions also exhibit the globalisation property, sometimes under additional assumptions which shall be discussed in Section \ref{sec:toponogov}.

In the manifold setting, Theorem \ref{thm: LorentzianToponogov} extends the work of Harris \cite{Har82} to globally hyperbolic Lipschitz spacetimes. 
This is due to \cite[Theorem 1.2]{GL18} by Graf and Ling, which shows that every strongly causal Lipschitz spacetime is a regular \LLSn, alongside the fact that a time function is guaranteed to exist on any second countable, globally hyperbolic \LLSn, see Proposition \ref{pop:exist-time-function}.
\medskip

We close this paper in Section \ref{sec: applicationsandoutlook} with an overview of some applications of our results. In particular, we extend the Lorentzian Bonnet--Myers Theorem \cite{BNR23} and Splitting Theorem \cite{BORS23} to Lorentzian length spaces with (local) lower curvature bounds, and show that lower curvature bounds are preserved under appropriate Lorentzian versions of Gromov--Hausdorff convergence (for example Minguzzi--Suhr convergence \cite{MS23}). Potential future results are also discussed.

\section{Preliminaries}\label{sec:preliminaries}

Over the course of the last half-decade, the theory of Lorentzian length spaces has gained immense traction, so much so that it is now a rather standard tool in the study of Lorentzian geometry. Consequently, in this section we only present material which is both critical for deriving our results and which also appears infrequently or disparately in the literature. In particular, we focus on the properties of hyperbolic angles \cite{BS22, BMS22}, time functions \cite{KS21}, and null distances \cite{SV16}. For more fundamental definitions, we refer the reader to \cite{KS18, BNR23}. 

\subsection{Notation and conventions}
\label{subsec: notationconventions}

Let us begin by reintroducing our main concepts and fixing our conventions. Recall that a \emph{Lorentzian pre-length space} $(X,d,\leq,\ll,\tau)$ consists of a metric space $(X,d)$ equipped with a causal relation $\leq$, timelike relation $\ll$, and time separation function $\tau$, cf.\ \cite[Definition 2.8]{KS18}. For brevity, we shall simply denote such spaces by their associated set $X$, where the additional structures can be identified from the context. A Lorentzian pre-length space which is additionally locally causally closed, causally path-connected, localisable, and whose time separation function takes the form
\begin{equation*}
\tau(x,y) = \sup \Set*{ L_\tau(\gamma) \given \gamma \textrm{ future-directed, causal curve from } x \textrm{ to } y}\, ,
\end{equation*}
for $x,y\in X$ with a future-directed causal curve between them and $\tau(x,y)=0$ otherwise, is called a \emph{Lorentzian length space}, cf. \cite[Definition 3.22]{KS18}.
\medskip

Unless explicitly stated otherwise, causal curves are assumed to be future-directed. Furthermore, we use the term \emph{distance realiser} to refer to any causal curve in a \LpLSn, cf.\ \cite[Definition 2.24]{KS18}, whose $\tau$-length attains the $\tau$-distance between its endpoints, i.e. a causal curve $\gamma$ from $x$ to $y$, such that $L_{\tau}(\gamma)=\tau(x,y)$. 
\medskip 

We inherit from earlier works the notion of the causal past/future of a point $x\in X$, which we denote by $J^\pm(x)$. The analogous timelike past/future is denoted $I^\pm(x)$. Causal and timelike diamonds with defining points $x,y\in X$ are respectively denoted by $J(x,y):=J^+(x) \cap J^-(y)$ and $I(x,y):=I^+(x) \cap I^-(y)$. Recall that a \LpLS is \emph{globally hyperbolic} if all causal diamonds $J(x,y)\subseteq X$ are compact and $X$ is non-totally imprisoning, cf\ \cite[Definition 2.35\,(iii)]{KS18}.
\medskip

We now wish to address the concept of regularity, one of the defining properties of a regularly localisable \LpLSn, cf.\ \cite[Definition 3.16]{KS18} and a natural condition to impose on a \LpLS in its own right. This property is also crucial for defining timelike curvature bounds via angle comparison.

\begin{dfn}[Regularity] \label{def: regular}
Let $X$ be a Lorentzian (pre-)length space. $X$ is called \emph{regular} if any distance realiser between timelike related points is timelike, i.e. it cannot contain a null piece.
\end{dfn}
It is worth observing that under strong causality, the notion of being regularly localisable is equivalent to being regular (in the sense of Definition \ref{def: regular}) and localisable, see \cite[Lemma 3.6]{BKR23}.

\subsection{Hyperbolic angles and curvature bounds}
\label{subsec:LLS:angles}
Hyperbolic angles in Lorentzian pre-length spaces were introduced in \cite{BS22} and \cite{BMS22}, where the latter puts a greater focus on comparison results. Throughout this section, we follow the conventions of the former reference.

First recall that the \emph{finite diameter} of a Lorentzian pre-length space is given by the supremum of (finite) $\tau$-values on the space. Denote by $\lm{K}$ the \emph{Lorentzian model space} of constant curvature $K$ and its finite diameter by $D_K$, cf.\ \cite[Definition 1.11]{BS22}. Similarly to the metric case, we have
\begin{equation*}
        D_K=\diamfin(\lm{K})= 
        \begin{cases}
            \infty, & \text{ if } K \geq 0 \, , \\
            \frac{\pi}{\sqrt{-K}}, & \text{ if } K < 0 \, .
        \end{cases}
\end{equation*} 
Furthermore, in a Lorentzian pre-length space, triples of points $(p,q,r)$ with $\tau(p,r)<\infty$, either $p\ll q\leq r$ or $p \leq q \ll r$, and (non-trivial) time-separations realised by distance realisers, will be called \emph{admissible causal triangles}. They shall be denoted by $\Delta(p,q,r)$, where the points are written according to their causal order unless otherwise stated, with each side being labelled either by the name of an associated distance realiser or, if the specific choice of distance realiser or parametrisation thereof is unimportant, by the closed interval between the endpoints, i.e. $[p,q]$ is a distance realiser from $p$ to $q$. If we additionally have  $p \ll q \ll r$, the triple is called a \emph{timelike triangle}, cf.\ \cite[Lemma 4.4]{KS18}. Throughout the remainder of this paper, we tacitly assume that any such triangles satisfy appropriate size bounds, cf.\ \cite[Lemma 4.6]{KS18}, that is, $\tau(p,r)<D_K$.

\begin{dfn}[Comparison angles]
Let $K\in \mathbb{R}$ and let $X$ be a \LpLSn. Let $x_1 \leq x_2 \leq x_3$ be a triple of causally related points in $X$, satisfying size bounds for $K$, cf.\ \cite[Lemma 4.6]{KS18} and let $\Delta(\bx_1, \bx_2, \bx_3)$ be a comparison triangle\footnote{Recall that a triple of causally related points has a comparison triangle in the model space $\lm{K}$ if the side-lengths satisfy size bounds with respect to $K$, cf.\ \cite[Definition 4.14]{KS18}. This does not require the points to be timelike related, nor that curves between the points exist.} in $\lm{K}$ for $(x_1,x_2,x_3)$. Fix distinct indices $i,j,k \in \{ 1,2,3 \}$ and assume that $x_i$ is timelike related to both $x_j$ and $x_k$ in some way. 
\begin{enumerate}
    \item The (unsigned) \emph{comparison angle} at $x_i$ is  
        \begin{equation*}
            \tilde\ma_{x_i}^K(x_j,x_k) \coloneqq \ma_{\bx_i}^{\lm{K}}(\bx_j,\bx_k)\,,
        \end{equation*}
        where $\ma_{\bx_i}^{\lm{K}}(\bx_j,\bx_k)$ is the hyperbolic angle at $\bx_i$ in $\Delta(\bx_1,\bx_2,\bx_3)\subseteq \lm{K}$.  
    \item The \emph{sign} $\sigma$ of the comparison angle at $x_i$ is defined to be $\sigma = 1$ if $i=2$, i.e. $x_i$ is not a time-endpoint, and $\sigma = -1$ if $i=1$ or $3$, i.e. $x_i$ is a time-endpoint.
    \item The \emph{signed comparison angle} at $x_i$ is then defined by 
        \begin{equation*}
            \tilde{\ma}_{x_i}^{\mathrm{S},K}(x_j,x_k)=\sigma\tilde{\ma}_{x_i}^{K}(x_j,x_k)\,,
        \end{equation*}
        where $\tilde{\ma}_{x_i}^K(x_j,x_k)>0$.
\end{enumerate}
\end{dfn}

The hyperbolic angle $\ma_{\bx_i}^{\lm{K}}(\bx_j,\bx_k)$ at $\bx_i$ in the Lorentzian model space $\lm{K}$ can be calculated using the law of cosines, cf.\ \cite[Lemma 2.3]{BS22}. For convenience, we reiterate this result here.

\begin{lem}[Law of cosines]
    Let $K \in \mathbb{R}$. Let $\bx_1\leq \bx_2 \leq \bx_3$ be a triple of causally related points in $\lm{K}$ forming a finite causal triangle. Fix distinct indices $i,j,k \in \{ 1,2,3 \}$ and assume that $\bx_i$ is timelike related to both $\bx_j$ and $\bx_k$ in some order. Denote the hyperbolic angle at $\bx_i$ by $\omega = \ma_{\bx_i}^{\lm{K}}(\bx_j,\bx_k)$, its sign by $\sigma$, and the scale factor by $s = \sqrt{|K|}$. Finally, set $a = \max \{ \tau(\bx_i,\bx_j),\tau(\bx_j,\bx_i) \}$, $b = \max \{ \tau(\bx_i,\bx_k),\tau(\bx_k,\bx_i)\}$, and $c = \max \{ \tau(\bx_j,\bx_k),\tau(\bx_k,\bx_j) \}$, noting that $a,b >0$ and $c \geq 0$. Then we have
    \begin{align*}
        &a^2 + b^2 = c^2 - 2ab\sigma \cosh(\omega) && \textrm{ for } K=0\,, \\
        &\cos(sc) = \cos(sa)\cos(sb) - \sigma \cosh(\omega) \sin(sa)\sin(sb) && \textrm{ for } K < 0\,, \\
        &\cosh(sc) = \cosh(sa)\cosh)sb) + \sigma \cosh(\omega)\sinh(sa)\sinh(sb) && \textrm{ for } K>0\,.
    \end{align*}
\end{lem}

An important consequence of the law of cosines is the following property for unsigned angles, which will be used extensively throughout this work. 

\begin{cor}[Law of cosines monotonicity]\label{cor:LOC Mono}
Let $K\in\mb{R}$ and consider any timelike triangle in the Lorentzian model space $\lm{K}$. 
Then fixing the two short side lengths and varying the longest, any angle is monotonically increasing. Fixing one short side and the longest side length and varying the other short side, any angle is monotonically decreasing.
\end{cor}

Both upper angles and angles between timelike curves in a Lorentzian pre-length space may now be defined via the comparison angle introduced above.

\begin{dfn}[Angles]
Let $X$ be a \LpLS and $\alpha,\beta:[0,\varepsilon)\to X$ be two timelike curves (where we permit one or both of the curves to be past-directed) with $x:=\alpha(0)=\beta(0)$. Then we define the \emph{upper angle} 
\begin{equation*}
\ma_x(\alpha,\beta)=\limsup_{\substack{(s,t)\in D \\ s,t\to 0}}\tilde{\ma}_x^{K}(\alpha(s),\beta(t))\,,
\end{equation*}
where 
\begin{align*}
D &= \Set*{ (s,t) \given s,t>0, \, \alpha(s),\beta(t) \text{ timelike related} } \\
& \qquad \cap \Set*{ (s,t) \given \alpha(s),\beta(t), x \text{ satisfies size bounds for $K$}}\,.
\end{align*}
If the limit superior is in fact a limit and is finite, we say the angle exists and call $\ma_x(\alpha,\beta)$ an \emph{angle}.

Observe that the sign $\sigma$ of the comparison angle is independent of $(s,t)\in D$. Therefore, the \emph{sign} of the (upper) angle is also defined to be precisely $\sigma$. The \emph{signed (upper) angle} is then defined as $\ma_x^{\mathrm{S}}(\alpha,\beta)=\sigma \ma_x(\alpha,\beta)$.
\end{dfn} 

The following proposition provides sufficient conditions for adjacent angles taken at a point along a distance realiser to be equal. This property is similar to the metric notion of a segment being balanced, cf.\ \cite[Lemma 1.3]{LS13}, and, as such, it will be crucial in constructing a proof of our main result.

\begin{pop}[Balanced segments in \LpLSn]
\label{pop: equal angles along geodesic}
Let $X$ be a strongly causal and locally causally closed \LpLS with timelike curvature bounded below by $K \in \mathbb{R}$, and let $\alpha:[0,1] \to X$ be a timelike distance realiser. Let $x=\alpha(t)$ for $t \in (0,1)$ and consider the restrictions $\alpha_-=\alpha|_{[0,t]}$ and $\alpha_+=\alpha|_{[t,1]}$ as past-directed and future-directed distance realisers emanating from $x$, respectively. 
Let $\beta$ be a timelike distance realiser emanating from $x$. Then $\ma_x(\alpha_-,\beta)=\ma_x(\alpha_+,\beta)$. 
\end{pop}
\begin{proof}
See \cite[Corollary 4.6]{BS22} (and \cite[Lemma 4.10]{BS22} for the existence of the angle).
\end{proof}

Throughout this paper, we make use of several different formulations of curvature bounds via comparison methods. Each of these has been introduced in the context of Lorentzian length spaces in earlier works, with full details on all current formulations being found in \cite{BKR23}, which also provides conditions under which they are equivalent. 
Since we predominantly use the formulation of curvature bounds in terms of angle comparison, we now provide this explicitly. This angle comparison condition is analogous to the one globalised by \cite{Har82} in the smooth Lorentzian setting and is the definition to which our globalisation result will directly apply. 

\begin{dfn}[Curvature bounds by angle comparison]
\label{def-cb-ang}
An open subset $U$ in a regular \LpLS $X$ is called a \emph{$(\geq K)$-comparison neighbourhood} if it satisfies the following:  
\begin{enumerate}
\item\label{def-cb-ang.item1} $\tau$ is continuous on $(U\times U) \cap \tau^{-1}([0,D_K))$ and this set is open.
\item\label{def-cb-ang.item2} For all $x,y \in U$ with $x \ll y$ and $\tau(x,y) < D_K$ there exists a distance realiser contained entirely in $U$ connecting $x$ and $y$.
\item \label{def-cb-ang.main} Let $\alpha:[0,a]\to U,\beta:[0,b]\to U$ be timelike distance realisers with arbitrary time-orientation and such that $x:=\alpha(0)=\beta(0)$ and $\Delta(x,\alpha(a),\beta(b))$, with some permutation of vertices, is an admissible causal triangle satisfying size bounds. 
Then 
\begin{equation*}
\ma_x^{\mathrm{S}}(\alpha,\beta)\leq\tilde{\ma}_x^{K,\mathrm{S}}(\alpha(a),\beta(b))\,.
\end{equation*}
\item\label{def-cb-ang.item4} Additionally, the following property must hold. If $\alpha,\beta,\gamma:[0,\varepsilon)\to U$ are three timelike curves with $x:=\alpha(0)=\beta(0)=\gamma(0)$, $\alpha,\gamma$ pointing in the same time direction, and $\beta$ in the other, then we have the following special case of the triangle inequality of angles: 
\begin{equation}
\label{eq: triangle inequality for angles for lower curvature bounds}
\ma_x(\alpha,\gamma)\leq\ma_x(\alpha,\beta)+\ma_x(\beta,\gamma)\, .
\end{equation}
\end{enumerate}
We say that $X$ has \emph{curvature bounded below by $K$ in the sense of angle comparison} if every point in $X$ has a $(\geq K)$-comparison neighbourhood. 

If $X$ itself is a $(\geq K)$-comparison neighbourhood, then we say that $X$ has \emph{curvature globally bounded below by $K$}, and similarly for curvature bounds above. 
\end{dfn}

Observe that, in point (iv) of the above definition, we can also take the curves to be maps into $X$, as the angles only depend on the initial segments of the curves. Furthermore, when considering curvature bounds from above, the inequality in (iii) is reversed and (iv) is dropped, though this notion will not be used in the remainder of the paper.
We make use of two other characterisations of curvature bounds in this paper: hinge comparison and triangle comparison.
The definition of triangle comparison does not include the triangle inequality of angles given in \eqref{eq: triangle inequality for angles for lower curvature bounds}, but the equivalence of all three characterisations (Proposition \ref{pop: equivalence of curvature bounds}) is at present only known under the assumption that \eqref{eq: triangle inequality for angles for lower curvature bounds} holds.
This is another reason why we prefer to work with angle comparison directly: \eqref{eq: triangle inequality for angles for lower curvature bounds} is already imposed by definition. 
We explain here how to adapt the definition of a $(\geq K)$-comparison neighbourhood to obtain these alternative characterisations.

Curvature is bounded below in the sense of \emph{hinge comparison} if item (iii) of Definition \ref{def-cb-ang} is replaced by the following statement:
Let $\alpha:[0,a]\to U,\beta:[0,b]\to U$ be timelike distance realisers with arbitrary time-orientation and such that $x:=\alpha(0)=\beta(0)$ and $\Delta(x,\alpha(a),\beta(b))$, with some permutation of vertices, is an admissible causal triangle satisfying size bounds. Then either $\ma_x^{\mathrm{S}}(\alpha, \beta) = - \infty$ or else $\ma_x(\alpha, \beta)$ is finite, in which case, letting $(\tilde{\alpha}, \tilde{\beta})$ form a comparison hinge in $\lm{K}$ for $(\alpha, \beta)$, 
\begin{equation*}
\tau(\alpha(a), \beta(b)) \geq \tau(\tilde{\alpha}(a), \tilde{\beta}(b))\,.
\end{equation*}

Curvature is bounded below in the sense of \emph{triangle comparison} if item (iv) of Definition \ref{def-cb-ang} is removed and item (iii)  is replaced by the following statement: 
Let $\Delta(x,y,z)$ be a timelike triangle in $U$, with $p$, $q$ two points on the sides of $\Delta(x,y,z)$. Let $\Delta(\bx, \by, \bz)$ be a comparison triangle in $\lm{K}$ for $\Delta(x,y,z)$ and $\bp$ and $\bq$ be comparison points for $p$ and $q$ respectively. Then
\begin{equation*}
\tau(p,q) \leq \tau(\bp, \bq)\,.
\end{equation*}

We now state the equivalence result for these three characterisations of curvature bounds. 
This result is an application of part of \cite[Theorem 5.1]{BKR23} to our specific setting.

\begin{pop}[Equivalence of curvature bounds]
\label{pop: equivalence of curvature bounds}
    Let $X$ be a globally hyperbolic and regular Lorentzian length space which satisfies \eqref{eq: triangle inequality for angles for lower curvature bounds}. 
    Then the following are equivalent for an open subset $U \subseteq X$: 
    \begin{enumerate}
        \item $U$ is a $(\geq K)$-comparison neighbourhood in the sense of angle comparison, cf.\ Definition \ref{def-cb-ang}. 
        \item $U$ is a $(\geq K)$-comparison neighbourhood in the sense of hinge comparison, cf.\ \cite[Definition 3.14]{BKR23}.
        \item $U$ is a $(\geq K)$-comparison neighbourhood in the sense of timelike triangle comparison, cf.\ \cite[Definition 3.1]{BKR23}.  
    \end{enumerate}
\end{pop}

Our eventual proof of the globalisation of timelike curvature bounds will consider admissible causal triangles which are not contained in comparison neighbourhoods and for which Definition \ref{def-cb-ang}\ref{def-cb-ang.main} fails to hold at some vertex and show that, under certain assumptions, these cannot exist. We formulate the aforementioned failure characteristic more precisely as follows.

\begin{dfn}[Angle condition holds/ fails]\label{def-failing-angle}
Let $X$ be a regular Lorentzian pre-length space with timelike curvature bounded below by $K\in \mathbb{R}$ in the sense of angle comparison and let $\alpha:[0,a]\rightarrow X$, $\beta:[0,b]\rightarrow X$ be timelike distance realisers of arbitrary time-orientation (not necessarily contained in a comparison neighbourhood), with $L(\alpha)$, $L(\beta)$, $\tau(\alpha(a), \beta(b))$, $\tau(\beta(b), \alpha(a)) < D_K$,
and such that $x\coloneqq \alpha(0)=\beta(0)$ and $\alpha(a),\beta(b)$ are causally related. We say that the \emph{angle condition holds} at $x$ if Definition \ref{def-cb-ang}\ref{def-cb-ang.main} is satisfied at $x$, with respect to the curvature bound $K$ on $X$. Similarly, we say that the \emph{angle condition fails to hold} at $x$ if Definition \ref{def-cb-ang}\ref{def-cb-ang.main} is not satisfied at $x$, i.e. if the inequality
\begin{equation*}
\ma_x^{\mathrm{S}}(\alpha,\beta)>\tilde{\ma}_x^{K,\mathrm{S}}(\alpha(a),\beta(b))\,,
\end{equation*}
holds, with respect to the curvature bound $K$ on $X$. In particular, the angle condition may be said to hold/fail at vertices between timelike sides of an admissible causal triangle. 
\end{dfn}

Moreover, note that by \cite[Remark 3.12]{BKR23}, it is sufficient to only consider timelike triangles when dealing with curvature bounds in the sense of angle comparison.

In order to verify whether or not triangles may have a failing angle condition, we need to be able to divide timelike triangles into smaller timelike triangles for which the answer to this question is known. 
To do so, we will use the Lorentzian versions of Alexandrov's Lemma. There are two very similar versions, each corresponding to a different subcase depending on which side we divide along; more precisely, the ``across version'' discusses divisions along the longest side, while the ``future version'' discusses divisions along one of the shorter sides. Since the statements of these lemmata are rather extensive, we only provide the statement of the latter. The former is illustrated in Figure \ref{fig: alexlem across concave} and the reader is referred to \cite[Proposition 2.42, 2.43]{BORS23} and \cite[Lemma 4.2.1, 4.2.2]{BR22} for more detail, including proofs of the respective statements. While the presentation in \cite{BORS23} concerns the case $K=0$, generalising to non-zero $K$ is straightforward, provided we assume the associated size bounds.

\begin{pop}[Alexandrov Lemma: future version]
\label{lem: alexlem future}
Let $X$ be a \LpLSn. 
Let $\Delta:=\Delta(p,q,r)$ be a timelike triangle satisfying size bounds for $K$. 
Let $x$ be a point on the side $[p,q]$, such that the distance realiser between $x$ and $r$ exists. 
Then we can consider the smaller triangles $\Delta_1:=\Delta(p,x,r)$ and $\Delta_2:=\Delta(x,q,r)$. We construct a comparison situation consisting of a comparison triangle $\bar{\Delta}_1$ for $\Delta_1$ and $\bar{\Delta}_2$ for $\Delta_2$, with $\bar{p}$ and $\bar{q}$ on different sides of the line through $[\bar{x},\bar{r}]$ and a comparison triangle $\tilde{\Delta}$ for $\Delta$ with a comparison point $\tilde{x}$ for $x$ on the side $[\tilde p,\tilde q]$. This contains the subtriangles $\tilde{\Delta}_1:=\Delta(\tilde{p},\tilde{x},\tilde{r})$ and $\tilde{\Delta}_2:=\Delta(\tilde{x},\tilde{q},\tilde{r})$, see Figure \ref{fig: alexlem future convex}.

\begin{figure}
\begin{center}
\begin{tikzpicture}
\draw (7,0)-- (6.236686420076234,1.0356870286413933);
\draw (6.236686420076234,1.0356870286413933)-- (7.548286699820673,4.0374024205174575);
\draw (7.548286699820673,4.0374024205174575)-- (7,0);
\draw (6.15561793604921,1.839784099337744)-- (6.236686420076234,1.0356870286413933);
\draw (6.15561793604921,1.839784099337744)-- (7.548286699820673,4.0374024205174575);
\draw (9.204717676576246,1.6977850199451887)-- (10,0);
\draw (9.204717676576246,1.6977850199451887)-- (10.90284745202584,4.1006259914346685);
\draw (10.90284745202584,4.1006259914346685)-- (10,0);
\draw [dashed] (9.628868249068917,0.7922996759744166)-- (10.90284745202584,4.1006259914346685);
\begin{scriptsize}
\coordinate [circle, fill=black, inner sep=0.7pt, label=270: {$\tilde{p}$}] (A1) at (10,0);
\coordinate [circle, fill=black, inner sep=0.7pt, label=180: {$\tilde{x}$}] (A1) at (9.628868249068917,0.7922996759744166);
\coordinate [circle, fill=black, inner sep=0.7pt, label=270: {$\bar{p}$}] (A1) at (7,0);
\coordinate [circle, fill=black, inner sep=0.7pt, label=180: {$\bar{x}$}] (A1) at (6.236686420076234,1.0356870286413933);
\coordinate [circle, fill=black, inner sep=0.7pt, label=90: {$\bar{r}$}] (A1) at (7.548286699820673,4.0374024205174575);
\coordinate [circle, fill=black, inner sep=0.7pt, label=180: {$\bar{q}$}] (A1) at (6.15561793604921,1.839784099337744);
\coordinate [circle, fill=black, inner sep=0.7pt, label=180: {$\tilde{q}$}] (A1) at (9.204717676576246,1.6977850199451887);
\coordinate [circle, fill=black, inner sep=0.7pt, label=90: {$\tilde{r}$}] (A1) at (10.90284745202584,4.1006259914346685);
\end{scriptsize}
\end{tikzpicture}
\end{center}
\caption{A convex situation in the future version of Alexandrov's Lemma. }
\label{fig: alexlem future convex}
\end{figure}
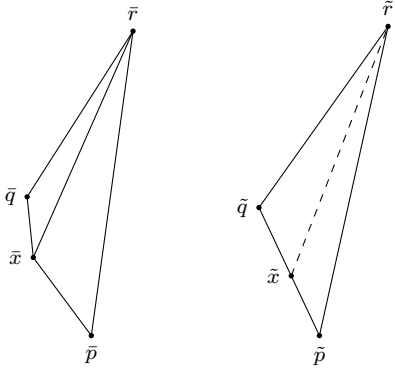

Then the situation $\bar{\Delta}_1$,$\bar{\Delta}_2$ is convex (concave) at $x$ (i.e.\
$\ma_{\bar{x}}(\bar{q},\bar{r})\leq\ma_{\bar{x}}(\bar{p},\bar{r})$ (or $\geq$)) if and only if 
$\tau(x,r) = \tau(\bx,\br) \leq \tau(\tilde{x},\tilde{r})$ (or $\geq$). 
The same is true if $x$ is a point on the side $[q,r]$. 
\end{pop}

Note that  if $X$ has timelike curvature bounded below (resp. above) by $K$ and $\Delta$ is within a comparison neighbourhood then 
$\tau (x,q)\leq \bar{\tau}(\tilde{x},\tilde{q})$ (resp. $\tau (x,q)\geq \bar{\tau}(\tilde{x},\tilde{q})$)
and so the convexity (resp. concavity) condition is always satisfied. 

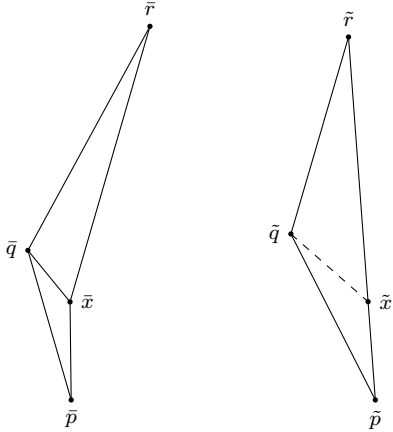
\begin{figure}
\begin{center}
\begin{tikzpicture}
\draw (-0.5693860319981044,1.9834819014638239)-- (0,0);
\draw (2.8896135929029714,2.2006721639230697)-- (4,0);
\draw (0,0)-- (-0.015287989182225509,1.3000898902049949);
\draw (-0.015287989182225509,1.3000898902049949)-- (-0.5693860319981044,1.9834819014638239);
\draw (-0.5693860319981044,1.9834819014638239)-- (1.0360567397613818,4.954583800321347);
\draw (1.0360567397613818,4.954583800321347)-- (-0.015287989182225509,1.3000898902049949);
\draw (2.8896135929029714,2.2006721639230697)-- (3.647225049614162,4.812946100427443);
\draw (3.647225049614162,4.812946100427443)-- (4,0);
\draw [dashed] (2.8896135929029714,2.2006721639230697)-- (3.904456784270502,1.303506235532433);
\begin{scriptsize}
\coordinate [circle, fill=black, inner sep=0.7pt, label=270: {$\bar{p}$}] (A1) at (0,0);
\coordinate [circle, fill=black, inner sep=0.7pt, label=0: {$\bar{x}$}] (A1) at (-0.015287989182225509,1.3000898902049949);
\coordinate [circle, fill=black, inner sep=0.7pt, label=180: {$\bar{q}$}] (A1) at (-0.5693860319981044,1.9834819014638239);
\coordinate [circle, fill=black, inner sep=0.7pt, label=90: {$\bar{r}$}] (A1) at (1.0360567397613818,4.954583800321347);
\coordinate [circle, fill=black, inner sep=0.7pt, label=270: {$\tilde{p}$}] (A1) at (4,0);
\coordinate [circle, fill=black, inner sep=0.7pt, label=180: {$\tilde{q}$}] (A1) at (2.8896135929029714,2.2006721639230697);
\coordinate [circle, fill=black, inner sep=0.7pt, label=90: {$\tilde{r}$}] (A1) at (3.647225049614162,4.812946100427443);
\coordinate [circle, fill=black, inner sep=0.7pt, label=0: {$\tilde{x}$}] (A1) at (3.904456784270502,1.303506235532433);
\end{scriptsize}
\end{tikzpicture}
\end{center}
\caption{A concave situation in the across version of Alexandrov's Lemma.}
\label{fig: alexlem across concave}
\end{figure}

\subsection{Null distance}

The null distance $d_T$ induced by a time function $T$ was originally introduced by Sormani and Vega \cite{SV16} in the smooth setting, as a convenient way of equipping a spacetime with a (distance) metric which is compatible with the causal structure. This concept has also been introduced in the setting of synthetic Lorentzian geometry, cf.\ \cite{KS21}, and is defined as follows: 

\begin{dfn}[Time Functions and Null Distance]\label{def:nullDistanceDefinition}
    Let $X$ be a \LpLSn .
    \begin{enumerate}
        \item A continuous map $T:X \to \rr$ is called a \emph{time function} if it is strictly monotonic with respect to the causal relation.
        \item A curve $\gamma: [a,b] \to X$ is called \emph{piecewise causal} if there exists a partition $a=s_1 \leq \ldots \leq s_k = b$ of $[a,b]$ such that $\gamma$ is causal or constant on each $[s_i, s_{i+1}]$.
        \item The \emph{null length} of a piecewise causal curve $\gamma: [a,b] \to X$ is 
            \begin{equation*}
                \hat{L}_T(\gamma) = \sum_{i=1}^k |T(\gamma(s_{i+1}) - T(\gamma(s_i)|\,.
            \end{equation*}
        \item The \emph{null distance} between two points $p$ and $q$ in $X$ is 
            \begin{equation}\label{def:nullDistance}
                d_T(p,q) = \inf \{ \hat{L}_T(\gamma) | \gamma \textnormal{ is piecewise causal from } p \textnormal{ to } q \} \,.
            \end{equation}    
    \end{enumerate}
\end{dfn}

In the case of spacetimes, if the infimum in \eqref{def:nullDistance} between non-timelike related point is achieved, this must occur along a piecewise null geodesic, cf.\ \cite[Lemma 3.20]{SV16}, inspiring the name. 
However, the null distance is not necessarily a true distance, with \cite[Theorem 4.6]{SV16} demonstrating that a necessary and sufficient condition for $d_T$ to be a distance function is $T$ being locally anti-Lipschitz. 
\medskip

With regard to our ultimate goal of globalisation, the null distance is also an ideal way of describing the ``size" of a timelike triangle. Contrary to the metric setting, there are always two notions of size at play in a Lorentzian pre-length space: on the one hand, we have the $\tau$-length of the sides of a triangle, which may be used to describe timelike curvature bounds, and on the other, we have the $d$-length of the sides, which is responsible for whether or not a triangle is inside a comparison neighbourhood. It will turn out that particularly well behaved null distances, when combined with timelike diamonds which are also comparison neighbourhoods, {\`a} la \cite[Proposition 4.3]{BNR23}, form the key to controlling both of these aspects simultaneously.
\medskip

Although in the next section we directly assume that our space possesses a time function, we first draw the reader's attention to the following result, which provides sufficient conditions for this to be the case.

\begin{pop}[Existence of time functions]\label{pop:exist-time-function}
Let $X$ be a second countable, globally hyperbolic Lorentzian length space. Then $X$ possesses a time function $T$.    
\end{pop}
\begin{proof}
The result is clear upon combining \cite[Theorem 3.2]{BGH21} with \cite[Theorem 3.20]{ACS20}, \cite[Lemma 3.8]{Rot22}, and \cite[Theorem 3.7]{KS18}. 
\end{proof}

We now wish to make the notion of a well behaved null distance more precise; in particular, we shall require our null distance to be a finite, continuous pseudo-metric.\footnote{By pseudo-metric we mean a metric which does not always distinguish points. Compare with the `semi-metric' applied to the quotient spaces in \cite{BR22,Rot22}.} If two points $p$ and $q$ are not connected by a piecewise causal curve, then $d_T(p,q) = \infty$. 
Therefore, if we require the null distance to be finite, it is necessary for there to be a piecewise causal curve between every pair of points. We begin by investigating when a space has this level of causal connectedness. 

\begin{lem}[Path-connected \LpLSsn]
\label{lem: path connected}
Let $X$ be a causally path connected Lorentzian pre-length space such that for each $x \in X$ either $I^+(x)$ or $I^-(x)$ is non-empty. Then the following are equivalent:
\begin{enumerate}
\item $X$ is connected.
\item $X$ is path connected.
\item $X$ is piecewise causal path connected, i.e.\ any $x,y\in X$ can be connected by a continuous curve consisting of future directed and past directed causal pieces, cf.\ \cite[Definition 3.2]{KS21}.
\end{enumerate}
\end{lem}

\begin{proof}
Two of the implications are clear, so let $X$ be connected and we claim it is piecewise causal path connected. 
Let $p\in X$ and $R_p$ be the set of all points which are connected to $p$ by piecewise causal paths. 
We claim that $R_p$ is open and in turn that $R_p = X$: By assumption, for each $q\in R_p$, there exists an $r\ll q$ (or $q\ll r$) and, as $X$ is causally path connected, a causal curve between them. 
Hence there is a piecewise causal curve from $p$ to $r$ and so $r\in R_p$. 
Similarly, each point in $I^+(r)$ (resp. $I^-(r)$) is connected to $r$ (and hence $p$) by a piecewise causal curve. So $I^+(r)\subseteq R_p$ (resp. $I^-(r)\subseteq R_p$) is an open neighbourhood of $q$ contained in $R_p$. As $q$ was arbitrary, it follows that $R_p$ is open. Furthermore, for any pair of points $p,q \in X$, the sets $R_p$ and $R_q$ are either equal or disjoint. 
Consequently, $\Set*{ R_p \given p \in X}$ gives an open partition of $X$. However, $X$ is connected, hence the partition must consist of precisely one element, namely $R_p = X$ for all $p\in X$, and $X$ is piecewise causal path-connected.
\end{proof}

It should be clear that the above lemma holds for Lorentzian length spaces and this is the context in which we will utilise the result. 
We also note that a \LpLS $X$ which is connected and causally path connected, such that for each $x\in X$ one of $I^+(x)$ or $I^-(x)$ is non-empty, is automatically \emph{sufficiently causally connected}, see \cite[Definition 3.4]{KS21}. The equivalence between path-connected and piecewise causal path-connected was also noted by \cite[Lemma 3.5]{KS21} and \cite[Lemma 3.5]{SV16} in their respective settings.
\medskip 

In the following proposition we demonstrate that the null distance on a connected Lorentzian length space satisfies all of the requirements of a distance function aside from separation of points, even if we do not assume that the associated time function is locally anti-Lipschitz (cf.\ \cite[Lemma 3.8]{SV16} for a corresponding result on spacetimes).

\begin{pop}[Null distance is a finite, continuous pseudo-metric]
\label{prop:pseudometric-null-distance}
Let $X$ be a connected Lorentzian length space with a (not necessarily locally anti-Lipschitz) time function $T$ and metric $d$. Then the null distance $d_T$, induced by $T$, is a finite pseudo-metric which is continuous (with respect to $d$). Moreover, 
\begin{equation}\label{eq: null-distance-time-function}
p \leq q \Rightarrow d_T(p,q)=T(q)-T(p)\, .
\end{equation}
\end{pop}

\begin{proof}
By our previous discussion, every connected Lorentzian length space is sufficiently causally connected.
The fact that $d_T$ is a finite pseudo-metric then follows directly from \cite[Lemma 3.7]{KS21}. Similarly, continuity of $d_T$ and \eqref{eq: null-distance-time-function} follow from \cite[Proposition 3.9]{KS21} and \cite[Proposition 3.8.(ii)]{KS21}, respectively.
\end{proof}

The diameter of a subset in a metric space is a well known concept, which also makes sense when considering such a pseudo-metric. 
We denote the $d_T$-diameter of a set by $\diam_T$. 
The following result on the diameter of causal diamonds improves the bound given in \cite[Proposition 3.8(iv)]{KS21} to an equality.

\begin{lem}[Null distance diameter of diamonds]\label{lem:CausalDiamondDiameter}
Let $X$ be a connected Lorentz\-ian length space with a (not necessarily locally anti-Lipschitz) time function $T$ and let $p\leq q$. 
Then $\diam_T(J(p,q)) = T(q)-T(p)$. If $p \ll q$, then also $\diam_T(I(p,q)) = T(q)-T(p)$. 
\end{lem}

\begin{proof}
By definition of the diameter, we always have $\diam_T(J(p,q)) \geq d_T(p,q)= T(q)-T(p)$. 
Let $x,y \in J(p,q)$. 
By applying the triangle inequality for $d_T$, we obtain from \eqref{eq: null-distance-time-function}
\begin{align*}
    d_T(x,y) &\leq d_T(x,p) + d_T(p,y) = T(x) - T(p) + T(y) - T(p) \\
    d_T(x,y) &\leq d_T(x,q) + d_T(q,y) = T(q) - T(x) + T(q) - T(y)\,.
\end{align*}
Summing these up, we obtain $2d_T(x,y) \leq 2(T(q)-T(p))=2d_T(p,q)$ and the claim follows.  

In the case $p \ll q$, we still have $\diam_T(I(p,q)) \leq \diam_T(J(p,q)) = T(q) - T(p)$. Conversely, consider sequences $p_n, q_n \in I(p,q)$ so that $p_n \to p$, $q_n \to q$, $p \ll p_n$, $q_n \ll q$, which exist by \cite[Lemma 2.25]{ACS20}. Then $d_T(p_n,q_n) = T(q_n) - T(p_n) \to T(q) - T(p) = d_T(p,q)$ since $T$ is continuous.
\end{proof}

Viewing an admissible causal triangle as the union of the images of the curves corresponding to its sides, we therefore have $\diam_T(\Delta(p,q,r)) = T(r) - T(p)$.
It is in this sense that the $d_T$-diameter of an admissible causal triangle gives a topological notion of its ``size'' which is more compatible with the causal structure. 
Of course, from a metric point of view, any admissible causal triangle is degenerate with respect to $d_T$, i.e. 
\begin{equation}
\label{eq: null distance degenerate in timelike triangles}
d_T(p,r)=d_T(p,q)+d_T(q,r)\,.
\end{equation}

In the next section we shall put the key we have just constructed to use and finally prove the Toponogov Globalisation Theorem for Lorentzian length spaces. 

\section{Lorentzian Toponogov Globalisation}\label{sec:toponogov}

The main goal of this section is to prove a synthetic Lorentzian analogue of Toponogov's Globalisation theorem for lower timelike curvature bounds. This will be proven in the setting of connected, globally hyperbolic, regular Lorentzian length spaces having a time function. As previously noted, second countability is sufficient for the existence of a time function. 
\medskip

However, before we dive into the proof proper, we first require a small collection of essential lemmata. To begin, recall that globally hyperbolic \LLSs $X$ are geodesic with finite and continuous time separation $\tau$ \cite[Theorems 3.28 and 3.30]{KS18}. Thus, in this case, \ref{def-cb-ang.item1} and \ref{def-cb-ang.item2} from Definition \ref{def-cb-ang} (curvature bounds in the sense of angle comparison) hold for $U=X$, i.e. globalisation of these properties is automatic for such spaces. We will also use the geodesic nature of globally hyperbolic Lorentzian length spaces implicitly throughout the remainder of this section, to avoid concerns regarding the existence of distance realisers.
\medskip 

Our next result is a slight adaptation of the Lebesgue Number Lemma, which allows us to properly configure coverings of causal diamonds by small and well behaved timelike diamonds.

\begin{lem}[Lebesgue Number Lemma, Lorentzian version]
\label{lem: Lebesgue number lemma}
Let $X$ be a connected, globally hyperbolic, Lorentzian length space with $T: X \to \R$ a time function on $X$ and let $d_T$ be the associated null distance. Consider any causal diamond $J(x,y)$ in $X$ and let $\{D_i\}_{i=1}^n$ be an open cover of $J(x,y)$ consisting of timelike diamonds.\footnote{Such an open cover must exist by \cite[Corollary 3.6]{Rot22} and \cite[Theorem 3.26.(v)]{KS18}.} Then there exists an $\varepsilon >0$ such that any causal (and hence any timelike) diamond with $d_T$-diameter less than $\varepsilon$ contained in $J(x,y)$ is also contained in one element of the covering.    
\end{lem}

\begin{proof}
The main difference when comparing to the original version of the Lebesgue number lemma is that $d_T$ is only a finite, continuous, pseudo-metric in general, as a result of Proposition \ref{prop:pseudometric-null-distance}. The causal structure of diamonds and its interplay with the null distance will be crucial in the proof. 

Firstly, if $J(x,y) \subseteq D_i$ for some $i$ then we can choose $\varepsilon$ arbitrary and we are done. 
Otherwise, denote by $C_i \coloneqq J(x,y) \setminus D_i$ the complement of $D_i$ in $J(x,y)$. Define a function $f: J(x,y) \to \R$ via 
\begin{equation}
f(p)=\max_{i\in\{ 1,2,\ldots, n\}} d_T(p, C_i \cap (J^+(p) \cup J^-(p))) \, .
\end{equation}
Note that the infimum in the definition of $d_T(p, C_i \cap (J^+(p) \cup J^-(p)))$ is attained as $C_i \cap (J^+(p) \cup J^-(p))$ is a closed subset of $J(x,y)$ and hence compact. We now show that $f(p) \in (0, \infty)$ for all $p$. 

If $f(p)$ were $0$ for some $p \in J(x,y)$, then for all $i$ there is a point $p_i \in C_i$ with $d_T(p, p_i) = 0$ such that $p_i \in J^+(p) \cup J^-(p)$. From this, we infer that $p = p_i$ and so $p \in C_i$ for all $i$.
Equivalently, $p \notin D_i$ for all $i$.
As the $D_i$ cover $J(x,y)$, we arrive at the contradiction $p \notin J(x,y)$. 
If $f(p)=\infty$ for some $p$, then there exists some $i$ such that $C_i \cap (J^+(p) \cup J^-(p)) = \emptyset$. Indeed, as all of these sets are compact and the null distance is finite valued, the maximum of finitely many infima can only be infinite if (at least) one of the sets is empty. 
Thus, $J(x,y) \cap (J^+(p) \cup J^-(p)) \subseteq D_i$, and hence $x,y \in D_i$. As $D_i$ is a timelike diamond and therefore causally convex, this implies $J(x,y) \subseteq D_i$, which we treated separately.

As the sets $C_i \cap (J^+(p) \cup J^-(p))$ are all compact and the null distance is continuous, it follows that $f$ is continuous and hence attains its minimum value. Consequently, set $\varepsilon \coloneqq \min_{p \in J(x,y)} f(p) >0$. Now let $p,q \in J(x,y)$ with $p \leq q$ and $\diam_T(J(p,q))=d_T(p,q) < \varepsilon$. As $f(p) \geq \varepsilon$, there exists $i$ such that $d_T(p, C_i \cap (J^+(p) \cup J^-(p)) \geq \varepsilon$. Then clearly, $p \notin C_i$. Furthermore, $p \leq q$ and $d_T(p,q) < \varepsilon$, hence also $q \notin C_i$. Thus, $p,q \in D_i$ and by the causal convexity of diamonds, also $J(p,q) \subseteq D_i$.
\end{proof}

We now turn to proving the most essential synthetic Lorentzian tool required for the proof of the Globalisation Theorem. Recall that the so-called Gluing Lemma for triangles with upper curvature bounds, \cite[Lemma 4.3.1, Corollary 4.3.2]{BR22},  roughly states that when two subtriangles satisfy the same curvature inequalities, then a large triangle formed by combining the two must also satisfy that curvature bound. The Gluing Lemma (and hence the Lorentzian analogue of the Reshetnyak Gluing Theorem \cite[Theorem 5.2.1]{BR22}) is not valid in full generality for lower curvature bounds, as not all of the inequalities in the Alexandrov Lemma \ref{lem: alexlem future} point in the same direction in this case.
\medskip

However, we propose the following result, in the spirit of the Gluing Lemma, under lower curvature bounds. In essence, if the angle condition fails to hold at a vertex in a timelike triangle then, upon splitting the triangle into two timelike subtriangles along one of the adjacent sides, then at least one angle condition must fail in one of the two subtriangles. In particular, the failing angle condition(s) will either be at the original vertex (viewed as part of a subtriangle), or at the point at which we split the adjacent side. 

\begin{lem}[Gluing Lemma for timelike triangles, lower curvature bounds]
\label{lem:triangleGluingCBB}
Let $X$ be a globally hyperbolic, regular Lorentzian length space with curvature bounded below by $K\in\mathbb{R}$ in the sense of angle comparison.
Let $\Delta(p,q,r)$ be a timelike triangle in $X$ (which is not necessarily contained in a comparison neighbourhood), where the sides are given by distance realisers $\alpha$ from $p$ to $r$, $\beta$ from $p$ to $q$ and $\gamma$ from $q$ to $r$, respectively.
Let $\Delta(\tp,\tq,\trr)$ be a comparison triangle for $\Delta(p,q,r)$ and assume that the angle condition fails to hold at $p$ in $\Delta(p,q,r)$, i.e. $\ma_p(\alpha, \beta) < \ma_{\tp}(\tq,\trr)$. 

Let $x$ be a point on $\beta$. 
Then at least one of the following three angle conditions fails to hold: the angle conditions at $x$ in $\Delta(p,x,r)$, at $p$ in $\Delta(p,x,r)$ and at $x$ in $\Delta(x,q,r)$ (see Figure \ref{fig: Gluing Lemma constellation}).
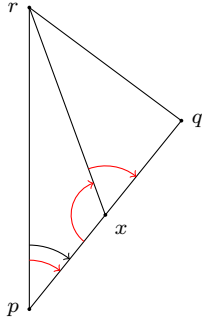
\begin{figure}
\begin{center}
\begin{tikzpicture}
\draw (0,0) -- (0,4) -- (2,2.5) -- (0,0);
\draw (0,4) -- (1,1.25);
\begin{scriptsize}
\coordinate [circle, fill=black, inner sep=0.5pt, label=180: {$r$}] (r) at   (0,4);
\coordinate [circle, fill=black, inner sep=0.5pt, label=0: {$q$}] (q) at   (2,2.5);
\coordinate [circle, fill=black, inner sep=0.5pt, label=300: {$x$}] (x) at   (1,1.25);
\coordinate [circle, fill=black, inner sep=0.5pt, label=180: {$p$}] (p) at   (0,0);
%angles
\pic [draw=black,<-, angle radius=8.5mm, angle eccentricity=.5] {angle = q--p--r};
\pic [draw=red,<-, angle radius=6.5mm, angle eccentricity=.5] {angle = x--p--r};
\pic [draw=red,<-, angle radius=6.5mm, angle eccentricity=.5] {angle = q--x--r};
\pic [draw=red,<-, angle radius=4.5mm, angle eccentricity=.5] {angle = r--x--p};
\end{scriptsize}
\end{tikzpicture}
\end{center}
\caption{If the angle condition at $p$ in $\Delta(p,q,r)$ fails to hold (in black), then at least one of the three angles conditions (in red) at $x$ or $p$ in the smaller triangles fail to hold.}
\label{fig: Gluing Lemma constellation}
\end{figure}

An analogous statement holds if $x$ is on $\alpha$ and timelike related to $q$, or if the angle condition initially failed at $r$ (and the subdividing point $x$ is on $\gamma$ or on $\alpha$ and timelike related to $q$) or at $q$ (and $x$ is on either $\beta$ or $\gamma$), instead of $p$. 
\end{lem}
\begin{proof}
We prove the result for the case where the angle condition fails to hold at $p$ in $\Delta(p,q,r)$ and $x$ is on $\beta$. Denote a distance realiser (which exists since $X$ is globally hyperbolic) from $x$ to $r$ by $\eta$.
Denote by $\beta_-$ and $\beta_+$ the parts of $\beta$ which go from $x$ to $p$ and from $x$ to $q$, respectively. 
Consider comparison triangles $\Delta(\bp,\bx,\br)$ and $\Delta(\bx,\bq,\br)$ for $\Delta(p,x,r)$ and $\Delta(x,q,r)$, respectively, as well as a comparison triangle $\Delta(\tp,\tq,\trr)$ for $\Delta(p,q,r)$. 
Assume that the angle condition at $p$ in $\Delta(p,x,r)$  holds, i.e. $\ma_p(\alpha, \beta_-) \geq \ma_{\bp}(\bx,\br)$, otherwise we are done. We now show that the angle condition at $x$ in $\Delta(p,x,r)$ or at $x$ in $\Delta(x,q,r)$ must fail. 
Let $\tx$ be the comparison point for $x$ in $\Delta(\tp,\tq,\trr)$ and consider the subtriangle $\Delta(\tp,\tx,\trr)$. 
$\Delta(\bp,\bx,\br)$ and $\Delta(\tp,\tx,\trr)$ have two sides of equal length, and for the angles at $\bp$ and $\tp$ we know
\begin{equation}
\ma_{\bp}(\bx,\br) \leq \ma_p(\alpha, \beta) < \ma_{\tp}(\tq,\trr) = \ma_{\tp}(\tx,\trr) \, .
\end{equation}
Thus, law of cosines monotonicity gives $\tau(x,r)=\tau(\bx,\br) > \tau(\tx,\trr)$
and so, by the Alexandrov Lemma \ref{lem: alexlem future},  
the comparison triangles $\Delta(\bp,\bx,\br)$ and $\Delta(\bx,\bq,\br)$ form a concave situation, i.e.
\begin{equation}
\label{eq: gluing lemma ineq}
\ma_{\bx}(\bp,\br) < \ma_{\bx}(\bq,\br) \, .     
\end{equation}
Moreover, by Proposition \ref{pop: equal angles along geodesic}, we have $\ma_x(\beta_-,\eta)=\ma_x(\eta,\beta_+)$. 
If the angle condition were to hold both at $x$ in $\Delta(p,x,r)$ and at $x$ in $\Delta(x,q,r)$, then we would have
\begin{equation}
\ma_{\bx}(\bp,\br) \geq \ma_x(\beta_-,\eta) = \ma_x(\eta,\beta_+) \geq \ma_{\bx}(\bq,\br) \, , 
\end{equation}
a contradiction to \eqref{eq: gluing lemma ineq}. Hence, the angle condition must fail at $x$ either in $\Delta(p,x,r)$ or $\Delta(x,q,r)$, if it does not fail at $p$ in $\Delta(p,x,r)$.

For the remaining cases, the proof is similar, upon using the appropriate version of the Alexandrov Lemma (cf.\ \cite[Lemma 4.2.1]{BR22} or \cite[Proposition 2.42]{BORS23}). 
\end{proof}

As should be clear from the proof, this gluing property also holds for strongly causal, locally causally closed, regular \LpLSs with curvature bounded below in the sense of angle comparison.

Using the previous lemmata, we can now prove two results which, when taken together, allow us to prove our main theorem.  One key difficulty in generalising globalisation to the Lorentzian setting is that splitting a timelike triangle along the longest side does not, in general, produce two timelike triangles. This issue is handled by the first result, which demonstrates that if any angle fails, it is always possible to assume that an angle of type $\sigma = +1$ fails. 

\begin{pop}[Failing angles can be assumed to be of type $\sigma=+1$]\label{prop:negativeAngles}
Let $X$ be a connected, globally hyperbolic, regular Lorentzian length space with time function $T$ and curvature bounded below by $K\in \mathbb{R}$ in the sense of angle comparison. Let $0 < \epsilon < 1$. Let $\Delta = \Delta(p,q,r)$ be a timelike triangle in $X$ which satisfies the size bounds for $K$ and for which the angle condition fails at some vertex. 

Suppose that the angle condition holds at each angle in every timelike triangle $\Delta(p', q', r')\subseteq J(p,r)$ with $d_T(p', r') \leq (1-\epsilon) d_T(p,r)$. 
Then there is at least one timelike triangle $\Delta(p'', q'', r'') \subseteq J(p,r)$ such that the angle condition fails at $q''$, i.e., $\ma_{q''}(p'', r'') > \ma_{\bq''}(\bp'',\br'')$. 
\end{pop}

\begin{proof}
Without loss of generality, assume that the angle condition in $\Delta$ fails at $p$ (the case where it fails at $r$ is analogous under reversal of the time orientation, while if it fails at $q$ the result is trivially satisfied).

Splitting the side $[p,q]$ into two pieces at some $x \in [p,q]$, say the $d_T$-midpoint, by Lemma \ref{lem:triangleGluingCBB} we get that either an angle condition fails at $x$ in $\Delta(p,x,r)$, in which case the result follows, or at either $p$ in $\Delta(p,x,r)$ or $x$ in $\Delta(x,q,r)$. In either of the two latter cases, we rename the triangle where the angle condition fails by $\Delta(p_1, q_1, r)$, with the angle condition now failing at $p_1$. (Both triangles may have a failing angle condition, in which case we may simply pick one at random.)
This procedure can be repeated arbitrarily many times (see Figure \ref{fig: final steps labeling toponogov}) and, if no positive angle fails at any stage, this will result in a sequence of pairs  $p_n\ll q_n$ on the side $[p,q]$ such that the angle conditions in $\Delta(p_n,q_n,r)$ fail at $p_n$. If the new subdivision point (which is either relabelled to $p_n$ or $q_n$)    
is always chosen to be the midpoint of the side $[p_{n-1},q_{n-1}]$ in the $d_T$ metric, then $d_T(p_n, q_n) \to 0$ and, since these points lie on the distance realiser $[p,q]$, it must be the case that $p_n$ and $q_n$ have a common limit point $p^* \in [p,q] $ with $p_n\nearrow p^*$ and $q_n\searrow p^*$.

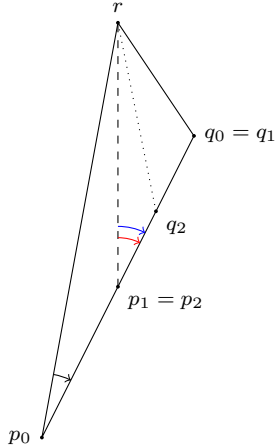
\begin{figure}
\begin{center}
\begin{tikzpicture}%[line cap=round,line join=round,>=triangle 45,x=1cm,y=1cm]
%triangle
\draw (0,0) -- (2,4) -- (1,5.5) -- (0,0);
\draw[dashed] (1,2) -- (1,5.5);
\draw[dotted] (1.5,3) -- (1,5.5);

\begin{scriptsize}
%points
\coordinate [circle, fill=black, inner sep=0.5pt, label=180: {$p_0$}] (p0) at (0,0); 
\coordinate [circle, fill=black, inner sep=0.5pt, label=0: {$q_0=q_1$}] (q0) at (2,4); 
\coordinate [circle, fill=black, inner sep=0.5pt, label=90: {$r$}] (r0) at (1,5.5);
%\coordinate [circle, fill=black, inner sep=0.5pt, label=90: {$*$}] (z2') at (-1.05,1.5); 

\coordinate [circle, fill=black, inner sep=0.5pt, label=330: {$p_1=p_2$}] (p1) at (1,2); 
\coordinate [circle, fill=black, inner sep=0.5pt, label=330: {$q_2$}] (q1) at (1.5,3); 

%angles
\pic [draw, <-, angle radius=8.5mm, angle eccentricity=2.5] {angle = q0--p0--r0};

\pic [draw, <-, color=red, angle radius=6.5mm, angle eccentricity=1.5] {angle = q0--p1--r0};

\pic [draw, <-, color=blue, angle radius=8mm, angle eccentricity=1.5] {angle = q0--p1--r0};
\end{scriptsize}

\end{tikzpicture}
\caption{The angle condition (black) originally fails to hold at $p_0$ in $\Delta(p_0,q_0,r)$. After the first subdivision (dashed), the angle condition (red) fails to hold at $p_1$ in $\Delta(p_1,q_1,r)$. After the second subdivision (dotted), the angle condition (blue) fails to hold at $p_2$ in $\Delta(p_2,q_2,r)$.}
\label{fig: final steps labeling toponogov}
\end{center}
\end{figure}

If $d_T(p^*, r) < (1-\epsilon) d_T(p,r)$, then $d_T(p_n, r) \leq (1-\epsilon) d_T(p,r)$ for large $n$ so that $\Delta(p_n, q_n, r)$ is already sufficiently small that it cannot have a failing angle, yielding a contradiction. However, this inequality need not hold and it may be necessary to split the long side $[p_n, r]$ in the following manner.

Let $r_n'$ be the point on the intersection of some distance realiser $[p_n, r]$ with $\partial J^+(q_n)$ (by regularity, this point of intersection is unique). 
By compactness of $J(p,r)$, we may, after passing to a subsequence if necessary, assume that $r_n'$ is convergent with $r_n' \to r^*$. 
By construction, $\tau(q_n,r_n')=0, q_n \leq r_n'$, and hence by continuity of $\tau$ and the closedness of the causal relation, we get $\tau(p^*,r^*)=0, p^* \leq r^*$. 
Moreover, we have $\tau(p_n,r)=\tau(p_n,r_n')+\tau(r_n',r)$ and hence again by continuity, $0<\tau(p^*,r)=\tau(p^*,r^*)+\tau(r^*,r)$, so the three points lie on a distance realiser. 
By regularity, the segment $[p^*, r]$ is timelike, so $\tau(p^*,r^*)=0 \implies p^*=r^*$.

For sufficiently large $n$, then, we may take a point $r_n$ slightly to the future of $r_n'$ on the segment $p_n r$. Then $p_n$, $q_n$ and $r_n$ are all so close to $p^*$ that the timelike triangle $\Delta(p_n,q_n,r_n)$ has $d_T(p_n, r_n) \leq (1-\epsilon) d_T(p,r)$. 
Splitting the triangle $\Delta(p_n,q_n,r)$, which has an angle condition failing at $p_n$, through $q_nr_n$ using Lemma \ref{lem:triangleGluingCBB}, results either in an angle condition failing at $r_n$ in $\Delta(q_n,r_n,r)$, so that the result follows, or at $p_n$ or $r_n$ in  $\Delta(p_n,q_n,r_n)$, which is not possible since  $\Delta(p_n,q_n,r_n)$ is sufficiently small in the $d_T$ metric.
\end{proof}

Following the work of Plaut across two papers \cite{Pla91,Pla96}, Lang and Schroeder \cite{LS13} provided a ``cat's cradle" construction for use in proving Toponogov's theorem for metric length spaces. Independently of and in parallel to this, Petrunin \cite{Pet16} also derived a similar, elegant scheme. In our second result, we demonstrate that this construction can also be used in the Lorentzian setting, despite the challenge posed by the fact that triangles with short side lengths (in $\tau$) need not be small topologically. This rules out the failure of angles of type $\sigma = +1$, provided that a collection of smaller triangles obey the angle condition at each of their vertices, essentially completing the proof.

\begin{pop}[Cat's cradle]\label{prop:catsCradle}
Let $X$ be a connected, globally hyperbolic, regular Lorentzian length space with time function $T$ and curvature bounded below by $K\in \mathbb{R}$ in the sense of angle comparison. Let $0 < \epsilon < \frac12$ and let $\Delta = \Delta(p,q,r)$ be a timelike triangle in $X$ which satisfies the size bounds for $K$. 

Suppose that the angle condition holds at each angle in every timelike triangle $\Delta(p', q', r')\subseteq J(p,r)$ with $d_T(p', r') \leq (1-\epsilon) d_T(p,r)$. 
Then the angle condition also holds at $q$ in $\Delta$. 
\end{pop}

Since the following proof is rather lengthy, we first offer a brief overview. 
The cat's cradle construction (see Figure \ref{fig:catsCradle}) is a recursive decomposition of $\Delta$ into smaller triangles designed to ensure that the angle condition holds for the $\sigma=+1$ angle opposite the longest side, namely for the angle at $q$. 
From this construction, we infer a sequence of inequalities \eqref{eq:LS5}. 
We then continue with a similarly recursive construction in the model space, assembling a sequence of comparison triangles to infer a sequence of inequalities \eqref{eq:LS6}. 
Finally we show that the two sequences of inequalities converge to the same limit, which implies that hinge comparison at $q$ cannot fail.

\begin{proof}
To begin, set $L\coloneqq d_T(p,r)$ and $q_0 \coloneqq q$. 
Assume without loss of generality that $d_T(p,q_0)\geq d_T(q_0,r)$, otherwise the roles of $p$ and $r$ should be interchanged. Let $q_1$ be the point on the distance realiser $[p,q_0]$ such that\footnote{By \eqref{eq: null distance degenerate in timelike triangles}, $d_T(p,q_0)\geq \frac12 L$ and as $\epsilon < \frac12$, it follows that $d_T(p,\cdot)$ attains $\epsilon L$ within the distance realiser $[p,q_0]$.}  $d_T(p,q_1)=\epsilon L$, from which it follows by \eqref{eq: null distance degenerate in timelike triangles} that $d_T(q_1,r)=(1 - \epsilon)L$. 
Now $\Delta_1 = \Delta(q_1, q_0, r)$ is a timelike triangle satisfying conditions of the statement, hence the angle condition holds at all vertices of $\Delta_1$ by assumption.

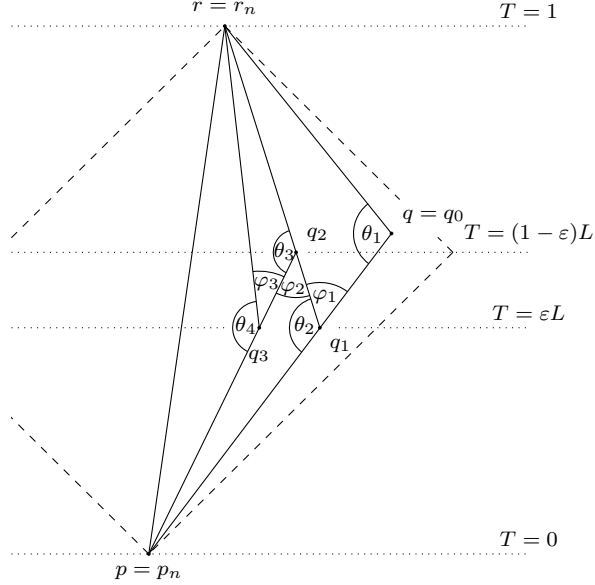
\begin{figure}[ht]
\begin{center}
\def\figCatsCradleEpsL{3}
\begin{tikzpicture}[label distance=0.2mm]
%[line cap=round,line join=round,>=triangle 45,x=1.0cm,y=1.0cm]
\clip(-2.80694926866508,-0.52257147695304372) rectangle (5.153415586204948,7.490647091159817);
\begin{scriptsize}
\draw[dotted](-5.,0.)-- (4.,0.) node[anchor=south]{$T=0$};%T=0
\draw[dotted] (-5.,\figCatsCradleEpsL)-- (0,\figCatsCradleEpsL) ;%T=\varepsilon L #1
\draw[dotted] (1.5,\figCatsCradleEpsL)-- (4.,\figCatsCradleEpsL) node[anchor=south]{$T=\varepsilon L$};%T=\varepsilon L #2
\draw[dotted] (-5.,{7.-\figCatsCradleEpsL})-- (0.55,{7.-\figCatsCradleEpsL}); %T=(1-\varepsilon) L #1
\draw[dotted] (1.1,{7.-\figCatsCradleEpsL})-- (4.,{7.-\figCatsCradleEpsL}) node[anchor=south]{$T=(1-\varepsilon)L$};%T=(1-\varepsilon) L #2

\draw[dotted] (-5.,7.)-- (4.,7.) node[anchor=south]{$T=1$};%T=1
\draw (-1.,0.)--(2.1925,4.25)--(0.,7.)--cycle;%pqr
\draw ({-1*(4-\figCatsCradleEpsL)/4+2*\figCatsCradleEpsL/4},\figCatsCradleEpsL)--(0.,7.);%q_1r
\draw (-1.,0.)--({(-1*1/4+2*3/4)*3/4+0*1/4},4.);%pq_2, UPDATE
\draw ({-1*1/4+((-1*1/4+2*3/4)*3/4+0*1/4)*3/4},3.) -- (0,7);%q_3r, UPDATE

%timelike diamond
\draw[dashed] (0.,7.)-- (2.3,4.7);
\draw[dashed] (3,4)-- (2.8,4.2);
\draw[dashed] (3.,4.)-- (-1.,0.);
\draw[dashed] (-1.,0.)-- (-4.,3.);
\draw[dashed] (-4.,3.)-- (0.,7.);

\coordinate [circle, fill=black, inner sep=0.5pt, label=270: {$p=p_n$}] (p) at (-1.,0.);

\coordinate [circle, fill=black, inner sep=0.5pt, label=90: {$r=r_n$}] (r) at   (0.,7.);

\coordinate [circle, fill=black, inner sep=0.5pt, label=45: {$q=q_0$}] (q0) at   (2.1925,4.25);

\coordinate [circle, fill=black, inner sep=0.5pt, label=300: {$q_1$}] (q1) at ({-1*(4-\figCatsCradleEpsL)/4+2*\figCatsCradleEpsL/4},\figCatsCradleEpsL);

\coordinate [circle, fill=black, inner sep=0.5pt, label=45: {$q_2$}] (q2) at ({(-1*1/4+2*3/4)*3/4+0*1/4},4.);
\coordinate [circle, fill=black, inner sep=0.5pt, label={[label distance=1.5mm]270: {$q_3$}}] (q3) at ({-1*1/4+((-1*1/4+2*3/4)*3/4+0*1/4)*3/4},3.);

%angles
\pic [draw=black, "$\theta_1$", angle radius=5mm, angle eccentricity=.5] {angle = r--q0--p};
\pic [draw=black, "$\theta_2$", angle radius=4mm, angle eccentricity=.5] {angle = r--q1--p};
\pic [draw=black, "$\theta_3$", angle radius=3mm, angle eccentricity=.5] {angle = r--q2--p};
\pic [draw=black, "$\theta_4$", angle radius=3.5mm, angle eccentricity=.5] {angle = r--q3--p};

\pic [draw=black, "$\phi_1$", angle radius=6mm, angle eccentricity=.7] {angle = q0--q1--r};
\pic [draw=black, "$\phi_2$", angle radius=6mm, angle eccentricity=.8] {angle = p--q2--q1};
\pic [draw=black, "$\phi_3$", angle radius=7.5mm, angle eccentricity=.8] {angle = q2--q3--r};
\end{scriptsize}
\end{tikzpicture}

\end{center}
\caption{The cat's cradle construction, showing the first three subtriangles $\Delta_1$, $\Delta_2$ and $\Delta_3$.}
\label{fig:catsCradle}
\end{figure}

We continue this construction recursively, picking points $q_n$, depending on whether $n$ is odd or even, to form new triangles. 
For even $n$, pick $q_{n}$ on the distance realiser $[q_{n-1},r]$ so that\footnote{Again, such a $q_n$ exists as $d_T(q_{n-1}, r) = (1-\epsilon) L>\varepsilon L$.} $d_T(q_{n}, r) = \epsilon L$ and $d_T(p, q_{n}) = (1 - \epsilon)L$. 
This defines a triangle $\Delta_{n} = \Delta(p, q_{n-1}, q_{n})$ for $n\geq 1$. 
Similarly, for odd $n$, pick $q_{n}$ on the distance realiser $[p,q_{n-1}]$ to define $\Delta_{n} = \Delta(q_{n}, q_{n-1}, r)$. 
In both cases, $\Delta_{n}$ satisfies the conditions of the statement and so the angle condition holds at all vertices of $\Delta_n$ by assumption.

Consider now the angles in $\Delta_n$. 
Let $\theta_{n} \coloneqq \ma_{q_{n-1}}(p, r)$ be the angle at $q_{n-1}$, which is given by $\ma_{q_{n-1}}(p,q_{n})$ or $\ma_{q_{n-1}}(q_{n},r)$ in $\Delta_{n}$, when $n$ is respectively even or odd.
Denote by $\phi_{n}$ the angle at $q_n$ in $\Delta_{n}$, which will be adjacent to $\theta_{n+1}$ in the subsequent triangle. 
When $n$ is even, $\phi_{n}$ is $\ma_{q_{n}}(q_{n-1}, p)$, while for odd $n$, the angle is $\ma_{q_{n}}(q_{n-1}, r)$. 
In either case, $ \phi_{n} =  \theta_{n+1}$, but with opposite signs $\sigma$, by Proposition \ref{pop: equal angles along geodesic}.

Set $l_n \coloneqq \tau(p, q_n) + \tau(q_n, r)$, for $n\geq 0$. By applying the reverse triangle inequality to each $\Delta_n$ (recalling that these are defined for $n\geq 1$), we have
\begin{equation}\label{eq:LS5}
0  <  l_0 \leq l_1 \leq \ldots \leq \tau(p, r)\,.
\end{equation}
Indeed, for odd $n$, we have $l_{n-1} = \tau(p, q_{n-1}) + \tau(q_{n-1}, r) = \tau(p, q_{n})+\tau(q_{n}, q_{n-1}) + \tau(q_{n-1}, r)\leq \tau(p, q_{n}) + \tau(q_{n}, r) = l_n$ and for even $n$ a similar argument can be used. 
The initial, strict inequality is due to $\Delta(p,q_0,r)$ being non-degenerate, while the final inequality in the chain follows from applying reverse triangle inequality to $\Delta(p,q_n,r)$.
The sequence $\{ (l_n) \}_{n\geq 0}$ in \eqref{eq:LS5} is a Cauchy sequence, as it is monotone increasing and bounded above by $\tau(p,r)$ (which is finite by size-bounds). Therefore, we have that $l_{n+1} - l_{n} \to 0$. 
This value is the excess in the triangle $\Delta_n$, that is, the value by which the longest side exceeds the sum of the two shortest sides (see Figure \ref{fig:catsCradle}). 
For $n$ even, this is $\tau(q_{n+1}, r) - \tau(q_{n}, r) - \tau(q_{n+1}, q_{n})$. For $n$ odd, on the other hand, this is $\tau(p, q_{n+1}) - \tau(p, q_{n}) - \tau(q_{n}, q_{n+1})$. 
\medskip

\textbf{Claim:} For some subsequence $n_i$, the time separation between the vertices $q_{n_i-1}$ and $ q_{n_i}$ of the triangle $\Delta_{n_i}$ is uniformly bounded away from zero. 

\textbf{Proof of claim:} 
For a contradiction, assume that the claim is false. 
Then we have $\lim_{n\to\infty} \tau(q_{2n-1}, q_{2n}) = \lim_{n\to\infty} \tau(q_{2n+1}, q_{2n}) = 0$. 
Consider the sequence of triples $\{ (q_{2n-1}, q_{2n}, q_{2n+1}) \}_{n\geq 1}$, which lies in the compact set $J(p,r)\times J(p,r)\times J(p,r)$. After passing to some subsequence $n_i$, we have that these converge to a limit triple $(q_a, q_b, q_c)$.  
Inspecting the time function, we see that
    \begin{equation*}
        T(q_{2n-1}) - T(p) = d_T(p,q_{2n-1}) = \epsilon L = d_T(p, q_{2n+1}) = T(q_{2n+1}) - T(p)\,.
    \end{equation*}
Furthermore, since $q_{2n-1} \leq q_{2n}$, \eqref{eq: null-distance-time-function} yields
    \begin{equation*}
        T(q_{2n}) - T(q_{2n \pm 1}) = d_T(q_{2n \pm 1}, q_{2n}) = (1-2\epsilon)L >0\,.
    \end{equation*}
Hence, $T(q_{2n-1})=T(q_{2n+1})\neq T(q_{2n})$, which in the limit $n \to \infty$ implies that $q_a\neq q_b\neq q_c$. 
Thus, by continuity of $\tau$, we have $\tau(q_a, q_b) = \tau(q_c, q_b) = 0$.

Again by continuity of $\tau$, we have $\tau(p,q_c)+\tau(q_c,q_b)=\tau(p,q_b)$ and by causal closedness, we have $p\leq q_c\leq q_b$. 
In particular, $p,q_c, \textrm{ and } q_b$ lie on a distance realiser with a non-constant null piece $[q_c,q_b]$. 
Thus, by regularity, the whole distance realiser must be null and therefore $\tau(p,q_b)=0$. 

Similarly, from $\tau(q_a,q_b)+\tau(q_b,r)=\tau(q_a,r)$ and $q_a\leq q_b\leq r$, we obtain that $q_a,q_b, \textrm{ and } r$ lie on a distance realiser which is null, so $\tau(q_b,r)=0$ (see Figure \ref{fig:catsCradleLimit}).
Therefore, $\lim_{i \to \infty} l_{2n_i} = \lim_{i \to \infty} \left( \tau(p, q_{2n_i})  + \tau(q_{2n_i}, r) \right) = \tau(p,q_b) + \tau(q_b,r) = 0$. 
However, \eqref{eq:LS5} states that $l_n$ is a non-decreasing sequence, beginning with $l_0 > 0$, which yields a contradiction. \textbf{Claim proven.}
\medskip

Let $p_n = p$ and $r_n = r$ for all $n\geq 0$. 
We now carry out a similar construction in the model space $\lm{K}$ by arranging comparison triangles $\bar{\Delta}_{n}$ (see Figure \ref{fig:catsCradleComparison}) for $\Delta_n$. 
Since, in general, the angles in $\bar\Delta_n$ do not match those in $\Delta_{n}$, the construction in $\lm{K}$ does not fit together as neatly. 

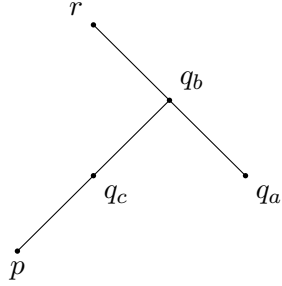
\begin{figure}
\begin{center}
\begin{tikzpicture}
\draw (-2,-2)-- (0,0);
\draw (1,-1) -- (-1,1);
\fill[color=black] 
    (-2,-2) circle(1pt) 
    (-1,-1) circle(1pt) 
    ( 0, 0) circle(1pt)
    ( 1,-1) circle(1pt)
    (-1, 1) circle(1pt);
\draw (-2,-2) node[anchor=north]{$p$} (-1,-1) node[anchor=north west]{$q_c$} (0,0) node[anchor=south west]{$q_b$} (1,-1) node[anchor=north west]{$q_a$} (-1,1) node[anchor=south east]{$r$};
\end{tikzpicture}
\end{center}
\caption{The limiting configuration of the cat's cradle, demonstrating that side lengths are bounded away from zero.}
\label{fig:catsCradleLimit}
\end{figure}

In fact, we begin by considering a comparison hinge $([\bq_0,\bp_0],[\bq_0,\br_0],\bar\omega_1)$ in $\lm{K}$ for $([q_0,p_0],[q_0,r_0],\theta_1)$; here, $(\bp_0,\bq_0,\br_0)$ is a triple of points such that $\tau(\bp_0,\bq_0) = \tau(p_0,q_0)$, $\tau(\bq_0,\br_0) = \tau(q_0,r_0)$, and the angle $\bar\omega_1$ between the distance realisers $[\bq_0,\bp_0],[\bq_0,\br_0]$ satisfies $\bar\omega_1 = \theta_1$. 
In particular, there is no a priori restriction on $\tau(\bp_0,\br_0)$ and instead we set out to obtain one (we are not considering a comparison triangle for $\Delta$, for example).

\begin{figure}
\begin{center}
\begin{tikzpicture}%[line cap=round,line join=round,>=triangle 45,x=1.0cm,y=1.0cm]
%\clip(-2,-2) rectangle (2,5);
%angles
\draw [shift={(1.5,2.)}] (0,0) -- (107.95575374797612:0.3905883856656444) arc (107.95575374797612:233.13010235415598:0.3905883856656444) -- cycle;
\draw [shift={(1.5,2.)}] (0,0) -- (130.50897178684028:0.8788238677476999) arc (130.50897178684028:233.13010235415598:0.8788238677476999) -- cycle;
\draw [shift={(0.,0.)}] (0,0) -- (53.13010235415598:0.7811767713312888) arc (53.13010235415598:94.45211030254173:0.7811767713312888) -- cycle;
\draw [shift={(0.,0.)}] (0,0) -- (94.45211030254175:0.5858825784984666) arc (94.45211030254175:223.65438458255832:0.5858825784984666) -- cycle;
\draw [shift={(-0.12869800373933468,1.6529252613237313)}] (0,0) -- (-114.46968190957956:0.9764709641641109) arc (-114.46968190957956:-85.54788969745829:0.9764709641641109) -- cycle;
\draw [shift={(-0.12869800373933468,1.6529252613237313)}] (0,0) -- (113.48433459211698:0.48823548208205547) arc (113.48433459211698:245.53031809042045:0.48823548208205547) -- cycle;

%triangles

\draw (1.5,2.)-- (-0.32174500934833666,4.132313153309328);
\draw (-0.32174500934833666,4.132313153309328)-- (0.,0.);
\draw (1.5,2.)-- (-0.9,-1.2);
\draw (1.5,2.)-- (1.0322265995494535,3.4434504049227366);
\draw (0.,0.)-- (-1.5569471788691276,-1.4854813292226978);
\draw (-1.5569471788691276,-1.4854813292226978)-- (-0.12869800373933468,1.6529252613237313);
\draw (-0.12869800373933468,1.6529252613237313)-- (-1.4381777394921285,4.66677563050582);
\draw (-1.4381777394921285,4.66677563050582)-- (-1.0678207490301574,-0.4106845516383042);
\begin{scriptsize}
%pts
\coordinate [circle, fill=black, inner sep=0.5pt, label=300: {$\bq_1$}] (q1) at (0.,0.);
\coordinate [circle, fill=black, inner sep=0.5pt, label=0: {$\bq_0$}] (q0) at (1.5,2.);
\coordinate [circle, fill=black, inner sep=0.5pt, label=90: {$\br_1=\br_2$}] (r1) at (-0.32174500934833666,4.132313153309328);
\coordinate [circle, fill=black, inner sep=0.5pt, label=300: {$\bp_0=\bp_1$}] (p0) at (-0.9,-1.2);
\coordinate [circle, fill=black, inner sep=0.5pt, label=90: {$\br_0$}] (r0) at (1.0322265995494535,3.4434504049227366);
\coordinate [circle, fill=black, inner sep=0.5pt, label=180: {$\bp_2$}] (p2) at (-1.5569471788691276,-1.4854813292226978);
\coordinate [circle, fill=black, inner sep=0.5pt, label=0: {$\bq_2$}] (q2) at (-0.12869800373933468,1.6529252613237313);
\coordinate [circle, fill=black, inner sep=0.5pt, label=180: {$\br_3$}] (r3) at (-1.4381777394921285,4.66677563050582);
\coordinate [circle, fill=black, inner sep=0.5pt, label=120: {$\bq_3$}] (q3) at (-1.0678207490301574,-0.4106845516383042);

\pic [draw=black, "$\bar{\phi}_3$", angle radius=10.1mm, angle eccentricity=.75] {angle = q2--q3--r3};

%angle labels
\coordinate [label=180: {$\bar{\theta}_1$}] (w1) at (1.2,2.);
\coordinate [label=180: {$\bar{\omega}_1$}] (t) at (1.55,2.);
\coordinate [label=85: {$\bar{\varphi}_1$}] (f1) at (-0.1,0.35);
\coordinate [label=180: {$\bar{\theta}_2$}] (w2) at (0,0.1);
\coordinate [label=180: {$\bar{\theta}_3$}] (w3) at (-0.12869800373933468,1.6529252613237313);
\coordinate [label=270: {$\bar{\varphi}_2$}] (f2) at (-0.2469800373933468,1.109252613237313);
\end{scriptsize}
\end{tikzpicture}
\end{center}
\caption{The comparison construction of the cat's cradle. Not marked for $n \geq 2$ are the angles $\bar{\omega}_n$ which are adjacent to $\bar{\phi}_{n-1}$, and are located in approximately the same position as $\bar{\theta}_n$.} 
\label{fig:catsCradleComparison}
\end{figure}
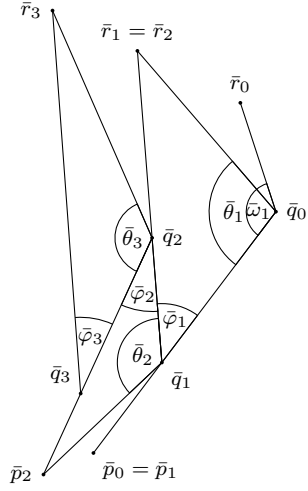

Using our hinge, we now recursively construct the comparison triangles $\bar\Delta_n$, for $n\geq 1$. 
For odd $n$, fix $\bar{q}_{n}$ on the distance realiser $[\bar{p}_{n-1},\bar{q}_{n-1}]$, such that $\tau(\bar{p}_{n-1},\bar{q}_{n}) = \tau(p_{n-1},q_{n})$. 
Then choose $\bar{r}_{n}$ such that the timelike triangle $\bar{\Delta}_{n} = \Delta(\bar{q}_{n},\bar{q}_{n-1}, \bar{r}_{n})$ has the same side lengths as $\Delta_{n}$. 
Finally, set $\bar{p}_{n} = \bar{p}_{n-1}$. 
For even $n$, similarly fix $\bar{q}_{n}$ on the distance realiser $ [\bar{q}_{n-1},\bar{r}_{n-1}]$, such that $\tau(\bar{q}_{n}, \br_{n-1}) = \tau(q_{n}, r_{n-1})$, construct a comparison triangle $\bar{\Delta}_{n} = \Delta(\bar{p}_{n}, \bar{q}_{n-1}, \bar{q}_{n})$, and set $\bar{r}_{n} = \bar{r}_{n-1}$.

The choice of the two new points at each stage again defines new angles. 
Denote by $\bar{\theta}_{n}$ the angle in $\bar{\Delta}_{n}$ at $\bar{q}_{n-1}$ (note that $\bar{\theta}_{n}=\tilde{\ma}_{q_{n-1}}(q_{n},r_{n})$ for $n$ odd and $\bar{\theta}_{n}=\tilde{\ma}_{q_{n-1}}(q_{n},p_{n})$ for $n$ even), by $\bar{\phi}_{n}$ the angle in $\bar{\Delta}_{n}$ at $\bar{q}_{n}$ and by $\bar{\omega}_{n+1}$ the angle of the remaining open hinge $([\bar{q}_{n},\bar{p}_{n}],[\bar{q}_{n},\bar{r}_{n}])$ adjacent to $\bar{\phi}_{n}$, see Figure \ref{fig:catsCradleComparison}. 
Note that $\bar{\phi}_{n} = \bar{\omega}_{n+1}$, but with opposite sign, again by Proposition \ref{pop: equal angles along geodesic}. 

As the angle condition holds at $q_{n-1}$ and $q_{n}$ in $\Delta_{n}$ by our assumptions, we have $\theta_{n} \leq \bar{\theta}_{n}$ at $q_{n-1}$, and at $q_{n}$, the type $\sigma = -1$ angle satisfies $\phi_{n} \geq \bar{\phi}_{n}$. 
Furthermore, by construction $\bar{\omega}_1 = \theta_1$ and by the above $\theta_1 \leq \bar\theta_1$, so $\bar{\omega}_1 \leq \bar\theta_1$. 
More generally, using the inequalities for $\phi_n$ and $\theta_n$ borne from the angle conditions holding in each $\Delta_n$, as well as equality of adjacent angles (see Proposition \ref{pop: equal angles along geodesic}), we obtain $\bar{\omega}_{n} = \bar{\phi}_{n-1} \leq \phi_{n-1} = \theta_{n} \leq \bar{\theta}_{n}$ for all $n\geq 2$. 
Therefore, we have $\bar\omega_n \leq \bar\theta_n$ for all $n \geq 1$, such that the relative sizes of the angles are indeed as depicted in Figure \ref{fig:catsCradleComparison}. 
Hence, by law of cosines monotonicity (Corollary \ref{cor:LOC Mono}), we have $\tau(\bar{p}_{n-1}, \bar{r}_{n-1}) \leq \tau(\bar{p}_{n}, \bar{r}_{n})$. 
Thus, the sequence of inequalities    
\begin{equation}\label{eq:LS6}
\tau(\bar{p}_0, \bar{r}_0) \leq \tau(\bar{p}_1, \bar{r}_1) \leq \ldots
\end{equation}
holds.

Consider again the subsequence $n_i$ from the claim above. Since in $\Delta_{n_i}$ the length of the (short) side $[q_{n_i-1},q_{n_i}]$  is uniformly bounded away from zero on this subsequence, the same is true of the length of the side $[\bq_{n_i-1},\bq_{n_i}]$ in $\bar\Delta_{n_i}$. Note that this implies the length of the longest side in $\bar\Delta_{n_i}$ is also uniformly bounded away from zero.
Hence, the angle $\bar\phi_{n_i}$ lies between two timelike sides of the triangle $\bar{\Delta}_{n_i}$ whose lengths are uniformly bounded away from zero, where the excess of $\bar{\Delta}_{n_i}$ (being equal to that of $\Delta_{n_i}$) is approaching $0$. 
This means that this sequence of configurations approaches a line, and not a point, so that $\bar\phi_{n_i}\to 0$. It follows from $\bar\omega_{n+1} = \bar\phi_n$ that $\bar\omega_{n_i+1}\to 0$. 
As $\bar\omega_{n_i+1}$ is given by $\ma_{\bar{q}_{n_i}}(\bar{p}_{n_i}, \bar{r}_{n_i})$, we conclude that, along our subsequence, $\tau(\bar{p}_{n_i}, \bar{r}_{n_i}) - \tau(\bar{p}_{n_i}, \bar{q}_{n_i}) - \tau(\bar{q}_{n_i}, \bar{r}_{n_i}) = \tau(\bar{p}_{n_i}, \bar{r}_{n_i}) - l_{n_i} \to 0$. 
In other words, the difference of the terms of the sequences in \eqref{eq:LS5} and \eqref{eq:LS6} is converging to 0.

Finally, assume that $\tau(p,r) < \tau(\bar{p}_0, \bar{r}_0)$, that is, the hinge condition \cite[Definition 3.14]{BKR23} fails at $q$ in $\Delta$. Set $C := \tau(\bar{p}_0, \bar{r}_0) - \tau(p,r)>0$. 
Since $\tau(\bar{p}_n, \bar{r}_n) - l_n \geq \tau(\bar{p}_0, \bar{r}_0) - \tau(p,r)$ for all $n \geq 0$ by \eqref{eq:LS5} and \eqref{eq:LS6}, we have $\tau(\bar{p}_n, \bar{r}_n) - l_n \geq C>0$ for all $n$, contradicting the fact that, on $n_i$, $\tau(\bar{p}_n, \bar{r}_n) - l_n \to 0$. It follows, therefore, that the hinge condition must hold at $q$ in $\Delta$. Since hinge comparison and angle comparison are equivalent (see Proposition \ref{pop: equivalence of curvature bounds}) the claim follows. 
\end{proof} 

Collecting the previous two propositions, we can deduce that the angle conditions hold in the large as long as they hold in the small. We formalise this statement here and will apply it in the proof of the main theorem.

\begin{cor}[Core argument of Lorentzian Toponogov Globalisation]\label{cor:coreArgument}
Let $X$ be a connected, globally hyperbolic, regular Lorentzian length space  with time function $T$ and curvature bounded below by $K\in \mathbb{R}$ in the sense of angle comparison. 
Let $0 < \epsilon < \frac12$. Let $\Delta = \Delta(p,q,r)$ be a timelike triangle in $X$ which satisfies the size bounds for $K$. 

Suppose that the angle condition holds at each angle in every timelike triangle $\Delta(p', q', r')\subseteq J(p,r)$ with $d_T(p', r') \leq (1-\epsilon) d_T(p,r)$.
Then the angle condition holds at all vertices of $\Delta(p',q',r')$, then the angle condition also holds at each angle in $\Delta$.
\end{cor}
\begin{proof}
First, observe that our assumptions include the criteria for Proposition \ref{prop:catsCradle} to hold. 
In particular, the angle condition must not fail at $q$ in $\Delta$. 
Now assume for a contradiction that the angle condition fails at either $p$ or $r$ in $\Delta$. 
Then by Proposition \ref{prop:negativeAngles}, there exists a timelike triangle $\Delta'' \coloneqq \Delta(p'',q'',r'') \subseteq J(p,r)$ such that the angle condition fails at $q''$.

Furthermore, by Lemma \ref{lem:CausalDiamondDiameter} and our discussion around \eqref{eq: null distance degenerate in timelike triangles}, we have $d_T(p'', r'') \leq d_T(p,r)$. Suppose that $\Delta(p', q', r') \subseteq J(p'', r'')$ is a timelike triangle with $d_T(p', r') \leq (1-\epsilon) d_T(p'',r'')$. 
Then it is also the case that $\Delta(p', q', r') \subseteq J(p, r)$ and $d_T(p', r') \leq (1-\epsilon) d_T(p,r)$. 
By the initial hypotheses, then, the angle condition holds at all vertices of all such $\Delta(p',q',r')$ and so Proposition \ref{prop:catsCradle} may be applied to $\Delta''$, to show that the angle condition cannot fail at $q''$, yielding a contradiction. 
Hence the angle condition must also not fail at $p$ or $r$ in $\Delta$ and our result follows.
\end{proof}

The previous result shows that the angle condition holds at all vertices of an arbitrarily large triangle, under the assumption that the angle condition holds for all vertices in a certain proportion of the smaller triangles in the space. It remains to show that (local) lower curvature bounds provide sufficiently many triangles with no failing angle condition for the above assumption to hold for each and every triangle. That is, no triangle possesses a vertex at which the angle condition fails.

\begin{thm}[Lorentzian Toponogov globalisation]\label{thm: LorentzianToponogov}
Let $X$ be a connected, globally hyperbolic, regular Lorentzian length space with a time function $T$ and curvature bounded below by $K\in \mathbb{R}$ in the sense of angle comparison. 
Then each of the properties in Definition \ref{def-cb-ang} hold globally; in particular, the entire space $X$ is a $(\geq K)$-comparison neighbourhood and hence has curvature globally bounded below by $K$. 
\end{thm}

\begin{proof}

First note that Definition \ref{def-cb-ang}\ref{def-cb-ang.item4} is a local condition, only requiring the germs of curves, hence it globalises trivially. Recall from the opening of this section that Definitions \ref{def-cb-ang}\ref{def-cb-ang.item1} and \ref{def-cb-ang}\ref{def-cb-ang.item2} also hold globally under our assumptions. 
It remains to check Definition \ref{def-cb-ang}\ref{def-cb-ang.main} for arbitrarily large triangles in $X$. 

Let $\Delta = \Delta(p,q,r)$ be a triangle in $X$, which we may assume to be timelike by \cite[Remark 3.12]{BKR23}, such that the angle condition fails at some vertex in $\Delta$ (this also permits triangles where the angle condition fails at multiple vertices). Clearly, $\Delta$ is contained in the causal diamond $J(p,r)$, which is compact by the global hyperbolicity of $X$. Suppose $\delta > 0$ is a greatest lower bound on the $d_T$-diameter of timelike triangles in $J(p,r)$ which exhibit a failing angle condition. In particular, any timelike triangle with $d_T$-diameter less than $\delta$ satisfies the angle condition, and there are triangles with $d_T$-diameter greater than yet arbitrarily close  to $\delta$ that exhibit a failing angle condition.\footnote{It is not strictly necessary that $\delta$ be a \emph{greatest} lower bound. This allows us to apply our propositions with an arbitrarily small constant $\epsilon > 0$, but they are stated for any $0< \epsilon < \frac12$.} Applying Corollary \ref{cor:coreArgument} to such triangles yields a contradiction which proves the result. 

All that remains is to establish the existence of the greatest lower bound $\delta$. Let $A$ be the set of $d_T$-diameters of triangles in $J(p,r)$ with a failing angle condition. By assumption, an angle condition fails in $\Delta(p,q,r)$, so $d_T(p,r) \in A$ and $A \neq \emptyset$. It follows that $A$ has a greatest lower bound, which we now verify is positive by demonstrating the existence of some positive lower bound.  By \cite[Proposition 4.3]{BNR23}\footnote{Recall that any globally hyperbolic Lorentzian length space is both non-timelike locally isolating and strongly causal. Furthermore, although \cite[Proposition 4.3]{BNR23} is formulated in terms of distance comparison, it is clear that the proof also holds for curvature bounds in terms of angle comparison.}, we can cover $J(p,r)$ by finitely many timelike diamonds which are all comparison neighbourhoods. Then by Lemma \ref{lem: Lebesgue number lemma} there exists some $\delta' >0$, such that any timelike diamond of $d_T$-diameter less than  $\delta'$ contained in $J(p,r)$ is contained in an element of this covering. In particular, any timelike triangle of $d_T$-diameter less than $\delta'$ is contained in a comparison neighbourhood and so has no failing angle conditions. It follows that $\delta'$ is a positive lower bound for $A$. 
\end{proof}

An application of Proposition \ref{pop: equivalence of curvature bounds} also yields that, provided \eqref{eq: triangle inequality for angles for lower curvature bounds} holds, lower curvature bounds in the sense of hinge, monotonicity, and triangle comparison also globalise. 

We also note that, since all causal diamonds $J(p,q)$ considered in this work satisfy $\tau(p,q) < D_K$, it is enough for only these diamonds to be assumed to be compact. The full power of global hyperbolicity is not used and we could instead use the weaker notion known as global hyperbolicity of order $\sqrt{-K}$ introduced by Harris \cite{Har82}.

\section{Applications and outlook}
\label{sec: applicationsandoutlook}

Finally, in this section we demonstrate the application of our results to the wider field of synthetic Lorentzian geometry and discuss potential refinements of the globalisation theorem along with some open problems.

\subsection{Gromov--Hausdorff convergence}

We begin by taking inspiration from the metric setting and consider the stability of curvature bounds under Gromov--Hausdorff convergence, a result which has been crucial for the proofs of finiteness results in Riemannian geometry.

Prior to the development of Alexandrov geometry as an independent subject, it was already understood that limits of Riemannian manifolds with sectional curvature bounded below are length spaces with curvature bounded below, in the sense that the conclusion of the Toponogov comparison theorem and certain nice topological properties hold \cite{GP91}. 
This insight was used to prove a variety of finiteness, pinching and rigidity results \cite{GP88, Yam88, GPW90, GP90, OSY89}.
The proof of the globalisation theorem for general Alexandrov spaces \cite{BGP92} placed this on a much clearer footing. 
It ensures that lower curvature bounds in the triangle comparison sense always survive Gromov--Hausdorff convergence, since there is no possibility that the size of comparison neighborhoods shrinks to zero along the sequence. 
Perelman used Alexandrov geometry to prove a much more powerful homeomorphism finiteness result for Alexandrov spaces and hence Riemannian manifolds \cite{Per91}, which has been generalised further to the setting of Riemannian orbifolds \cite{Harv16}.
\medskip

Gromov--Hausdorff convergence is most natural in the compact setting and can then be generalised to the non-compact case. 
As most interesting Lorentzian examples are non-compact, however, it has proved difficult to establish a general notion of convergence in this setting. Minguzzi and Suhr have provided an excellent notion of convergence for ``bounded Lorentzian metric spaces'' \cite{MS23} and in the globally hyperbolic case this can be applied to causal diamonds, as we will soon show. 
Indeed, since our work was completed, Bykov, Minguzzi and Suhr generalised this notion to the unbounded case \cite{BMS24}.
Sakovich and Sormani have also carried out extensive research into notions of intrinsic distances between ``causally-null-compactifiable spacetimes'' \cite{Sak04}.
We do not address these more recent notions of convergence in the present work.

For any reasonable notion of Gromov--Hausdorff convergence of Lorentz\-ian length spaces, we should expect that the condition of a timelike lower curvature bound is stable. 
This general principle is illustrated by Theorem \ref{thm: GHStability}, which brings together the globalisation result for spaces in the Kunzinger--S\"amann sense with the convergence result for bounded spaces in the Minguzzi--Suhr sense.
\medskip

A \emph{bounded Lorentzian metric space} is a topological space with a continuous time separation function satisfying a boundedness property ($\lbrace (p,q) : \tau(p,q) \geq \varepsilon \rbrace$ is compact for all $\varepsilon > 0$) and distinguishing points (if $p \neq q$ then for some $r$ either $\tau(p,r) \neq \tau(q,r)$ or $\tau(r,p) \neq \tau(r,q)$). 
It is a \emph{bounded Lorentzian length space} in the sense of Minguzzi--Suhr if timelike related points are connected by distance realisers (``maximal isocausal curves'' in the terminology of \cite{MS23}).

For bounded Lorentzian metric spaces, a Gromov--Hausdorff semi-distance can be defined simply by using the time separation in place of a metric. 
Bounded Lorentzian metric spaces admit at most one point which is not timelike related to any other point. 
If this exists it is denoted by $i_0$ and is called the spacelike boundary.
The Gromov--Hausdorff semi-distance, restricted to bounded Lorentzian metric spaces which do not contain $i_0$, is a true metric.
The same is true of bounded Lorentzian metric spaces which \emph{do} contain $i_0$.
For both classes of bounded Lorentzian metric spaces, then, we may speak of a Lorentzian Gromov--Hausdorff convergence.

We begin with a lemma to show that causal diamonds are bounded \LLSs in the Minguzzi--Suhr sense (after removing the spacelike boundary). 
Note, however, that causal diamonds are \emph{not} \LLSs in the Kunzinger--S\"amann sense, since they are not localisable.

\begin{lem}[Bounded \LLSs and causal diamonds]
    Let $X$ be a globally hyperbolic, regular \LLS (in the sense of Kunzinger--S\"amann, as used throughout this paper) and let $J(p,q)$ be a causal diamond in X. 
    Let $S$ be the set of points in $J(p,q)$ which are not timelike related to any other point in $J(p,q)$ -- the ``spacelike boundary'' of the diamond.
    Then $J(p,q) \setminus S$ is a bounded \LLS in the sense of Minguzzi--Suhr.
\end{lem}

\begin{proof}
Let $J(p,q)$ be a causal diamond in a globally hyperbolic \LLSn.  
By global hyperbolicity, $\tau$ is continuous with respect to the metric topology and, since $J(p,q)$ is compact and $\tau$ vanishes on $S$, the boundedness property holds on $J(p,q) \setminus S$. The final requirement for $J(p,q) \setminus S$ to be a bounded Lorentzian metric space is that $\tau$ distinguishes points. 

We adapt the argument from \cite{ACS20} which shows that globally hyperbolic \LLSs have the stronger property of being past- and future-distinguishing. 
Assume for a contradiction that there exist distinct points $x, y \in J(p,q)\setminus S$, which are not distinguished by $\tau$. 
In particular, $I^-(x) = I^-(y)$ and $I^+(x)=I^+(y)$.
If the points are timelike related to each other, this contradicts chronology, which is implied by global hyperbolicity. 

Consider now the case when $x$ and $y$ are not timelike related.
Since $x \notin S$, at least one point in $J(p,q) \setminus S$ is timelike related to $x$. 
Then, $x$  is joined to that point by a timelike curve in $J(p,q)$ and so is the limit of some sequence $x_n$, with the entire sequence lying either in $I^-(x)$ or $I^+(x)$.
Without loss of generality, suppose $x_n \in I^-(x)$.
Since $I^-(x) = I^-(y)$, we also have $x_n \in I^-(y)$. 
Hence, $x \in J^-(y)$, with $\tau(x,y)=0$ and $x\neq y$.
As $\tau$ does not distinguish $x$ and $y$, we have $\tau(x_n, x) = \tau(x_n, y) >0$, from which it follows that the concatenation of the distance realisers from $x_n$ to $x$ and from $x$ to $y$ forms a distance realising curve of mixed causal character, contradicting regularity.
Therefore $J(p,q) \setminus S$ is a bounded Lorentzian metric space.

Finally, we demonstrate that timelike related points are connected by distance realisers lying inside $J(p,q) \setminus S$.
Let $x \ll y$ in $J(p,q) \setminus S$.
By global hyperbolicity, $x$ and $y$ are connected by a distance realiser lying in $J(p,q)$.
If there were a point $s\in S$ on the distance realiser from $x$ to $y$, we would have $0=\tau(x, s) =\tau(s, y)$. 
However, $\tau (x, y) = \tau(x, s) + \tau(s, y) = 0$, contradicting $x \ll y$.
\end{proof}

\begin{thm}[Stability of lower curvature bounds]\label{thm: GHStability}
    Let $X_i$ be a sequence of connected, globally hyperbolic, regular, Lorentzian length spaces with time functions and curvature bounded below by $K\in \mathbb{R}$ in the sense of angle comparison. Let $J_i = J(p_i, q_i)$ be a sequence of causal diamonds in $X_i$ and let $S_i$ be the spacelike boundary of $J_i$.
    If the sequence $J_i \setminus S_i$ converges in the sense of Minguzzi--Suhr to some $J$, then $J$ is a bounded \LLS with sectional curvature bounded below by $K$ in the sense of Minguzzi--Suhr.
\end{thm}

\begin{proof}
    Each $J_i \setminus S_i$ is a bounded \LLS in the sense of Minguzzi--Suhr, by the previous lemma.
    By Theorem \ref{thm: LorentzianToponogov}, these spaces have a global lower curvature bound in any of the senses mentioned in Proposition \ref{pop: equivalence of curvature bounds}. In particular, $J_i \setminus S_i$ has curvature globally bounded below by $K$ in the sense of timelike triangle comparison, which is precisely the definition of sectional curvature bounded below by $K$ in the sense of Minguzzi--Suhr.
    By \cite[Theorem 5.18]{MS23}, the limit $J$ is a bounded \LLS in the sense of Minguzzi--Suhr.
    and by \cite[Theorem 6.7]{MS23}, it has sectional curvature bounded below in the sense of Minguzzi--Suhr. 
\end{proof}

In particular, an application of Proposition \ref{pop: equivalence of curvature bounds} also yields that, provided \eqref{eq: triangle inequality for angles for lower curvature bounds} holds on each $X_i$, lower curvature bounds in the sense of hinge, monotonicity, and triangle comparison are also stable under convergence, in the same sense, i.e.\ the limit space has sectional curvature bounded below in the sense of Minguzzi--Suhr.

\subsection{Geometric consequences}

There are also several direct corollaries to Theorem \ref{thm: LorentzianToponogov}, which extend known results for spaces with global timelike curvature bounds to those with local timelike curvature bounds, under the assumptions of our Toponogov-style Globalisation Theorem. 
In what follows, we present two such results, namely the Bonnet--Myers Theorem and the Splitting Theorem.
\medskip

First proven by Bonnet in two dimensions, the Bonnet--Myers theorem states that a complete Riemannian manifold with sectional curvature bounded below by some \emph{positive} $k\in \mathbb{R}$, has diameter $\diam(M) \leq \frac{\pi}{\sqrt{k}}$. 
For dimensions greater than two, the result was formalised by Myers \cite{Mye35}, who later demonstrated that the weaker assumption of a positive lower Ricci curvature bound was sufficient to obtain an associated upper bound on the diameter \cite{Mye41}. 
A corresponding synthetic result appears in \cite[Theorem 10.4.1]{BBI01}, where complete metric length spaces with sectional curvature bounded below by some $k>0$ are shown to also satisfy $\diam(X)\leq \frac{\pi}{\sqrt{k}}$.

Bonnet--Myers-style theorems also appear in the literature of Lorentzian geometry. 
In the smooth setting, Beem and Ehrlich \cite[Theorem 9.5]{BE79} have shown that globally hyperbolic spacetimes with timelike (sectional) curvature bounded below by some \emph{negative} $K \in \mathbb{R}$ have $\diam(M) \leq \frac{\pi}{\sqrt{-K}}$, where the diameter is now defined in terms of the Lorentzian distance function induced by the spacetime metric.\footnote{Lorentzian distance functions are, in essence, time-separation functions which are induced by a Lorentzian metric, in much the same way that a Riemannian manifold induces a distance.} In the synthetic Lorentzian setting, where the diameter is defined in terms of the time-separation function $\tau$, Cavalletti and Mondino \cite[Proposition 5.10]{CM24} have shown that measured Lorentzian pre-length spaces with suitable timelike measure contraction property (such as that implied by a lower Ricci curvature bound), also have an upper bound on their diameter. 
\medskip

Observe how, while the metric theorems consider $k>0$, the Lorentzian results concern $K<0$. This is not quite as superficial a change as it might first seem; it is a consequence of the hierarchy of curvature bound implications being reversed, following the conventions set by \cite{KS18}. 
In particular, in the metric setting, curvature bounded below by $k$ implies curvature bounded below by all $k'\leq k$, whereas in the Lorentzian setting, curvature bounded below by $K$ implies curvature bounded below by all $K' \geq K$. A similar statement holds for upper curvature bounds, with the inequalities reversed. 
Although we adhere to these conventions throughout this paper, they are by no means ubiquitous. For example, \cite{CM24, BE79, AB08} present Lorentzian results using the metric hierarchy.

While, in the metric setting, we could be content with a result utilising bounds on the Ricci curvature, since they are known to be weaker than sectional curvature bounds, see \cite{Pet19}, in the setting of Lorentzian pre-length spaces, the hierarchy of Ricci curvature bounds and timelike (sectional) curvature bounds via triangle comparison is an open question. 
As such, in \cite[Theorem 4.11]{BNR23}, a preliminary Bonnet--Myers result for timelike curvature bounds via triangle comparison is proven; namely, it is shown that strongly causal, locally causally closed, regular, and geodesic Lorentzian pre-length spaces with timelike curvature \emph{globally} bounded below by $K<0$ have finite diameter $\diam_{\mathrm{fin}}(X)\leq \frac{\pi}{\sqrt{-K}}$. Applying Theorem \ref{thm: LorentzianToponogov} re-frames this result in terms of local timelike curvature bounds as follows.
 
\begin{thm}[Synthetic Lorentzian Bonnet--Myers]
\label{thm: lor meyers}
Let $X$ be a connected, globally hyperbolic, and regular Lorentzian length space which has a time function $T$ and local curvature bounded below by $K\in \mathbb{R}$ in the sense of angle comparison. Assume $K<0$. Assume that $X$ possesses the following non-degeneracy condition: for each pair of points $x\ll z$ in $X$ we find $y \in X$ such that $\Delta(x,y,z)$ is a non-degenerate timelike triangle.
Then the diameter\footnote{Here we can replace the finite diameter with the diameter, since these notions coincide on globally hyperbolic Lorentzian length spaces.} satisfies $\diam(X)\leq \frac{\pi}{\sqrt{-K}}$. 
\end{thm}

Following \cite[Remark 4.12]{BNR23}, this result may be viewed as a direct synthetic extension of \cite[Theorem 9.5]{BE79}, with an additional non-degeneracy condition. Similarly to the exclusion of spaces isomorphic to $\mathbb{R}$, $(0,\infty)$, $[0,B]$ for all $B>\frac{\pi}{\sqrt{k}}$, and circles of radius greater than $\frac{1}{\sqrt{k}}$ in the metric setting, the non-degeneracy condition excludes locally one-dimensional spaces from the scope of our theorem. 
\medskip 

Recall that, throughout this paper, we have assumed triangles satisfy appropriate size-bounds, such that their comparison triangle is realisable cf.\ \cite[Lemma 2.1]{AB08}. In particular, given a Lorentzian pre-length space $X$ with curvature bounded below (or above) by $K$, we assume that triangles $\Delta(p,q,r)$ have $\tau(p,r)<D_K$. The following lemma, which was previously presented in the context of spacetimes by \cite[Proposition 9.4]{BE79}, gives us conditions under which the diameter of a Lorentzian pre-length space is not attained. 
Note that the following lemma is formulated via the ordinary diameter instead of the finite diameter, i.e. the supremum of all $\tau$-values in the space. 

\begin{lem}
    Let $X$ be a strongly causal Lorentzian pre-length space. If $\diam(X)$ is finite, then it is not attained on $X$. 
    Furthermore, if $X$ is a globally hyperbolic Lorentzian length space, then $\diam(X)$ is never attained on $X$, independently of whether it is finite.
\end{lem}
\begin{proof}
    Let $X$ be a strongly causal Lorentzian pre-length space. Assume for contradiction that $\diam(X)$ is finite and attained by some $p,q\in X$, that is, $\tau(p,q) = \diam(X)$. Then, by strong causality, there exists a point $q'$ with $q \ll q'$, such that $\tau(p,q') \geq \tau(p,q) + \tau(q,q') > \tau(p,q) = \diam(X)$, contradicting the definition of the diameter. 

    Now assume that $X$ is a globally hyperbolic Lorentzian length space. Recall that, on such a space, the time separation function is finite. Furthermore, the assumptions of the previous part still hold, hence $\diam(X)$ can never be attained.    
\end{proof}

Therefore, all triangles in Lorentzian pre-length spaces which satisfy the assumptions of either \cite[Theorem 4.11]{BNR23} or \ref{thm: lor meyers} for some $K<0$ satisfy size bounds for $K$. 
\medskip

Let us now move on to discussing the Splitting Theorem. Under the assumption of non-negative curvature, splitting theorems have also been proven in a variety of settings. In Riemannian geometry, Toponogov showed that if a complete manifold with non-negative sectional curvature contains a line, it splits as a product, with $\rr$ as one of the factors \cite{Top59-2, Top64}. Cheeger and Gromoll generalised this to the case where the manifold has only non-negative Ricci curvature \cite{CG71}. 

Beem, Ehrlich, Markvorsen and Galloway proved an analogous result for Lorentzian manifolds, where the hypothesis of completeness is replaced with global hyperbolicity, non-negative sectional curvature need only hold on timelike planes, and the line must be timelike \cite{BEMG85-2,BEMG85}. The assumption of non-negative (sectional) curvature can again be weakened to a bound on the Ricci curvature, known in general relativity as the strong energy condition. In increasing degrees of generality, Galloway~\cite{Gal84} and Eschenburg~\cite{Esc88} were able to prove this result for globally hyperbolic and timelike geodesically complete spacetimes, with Galloway dropping the latter assumption in~\cite{Gal89}. The splitting theorem for timelike geodesically complete spacetimes, as originally conjectured by Yau in~\cite{Yau82} (without the assumption of global hyperbolicity), was finally proven by Newman in~\cite{New90}. 
\medskip

In the synthetic setting, Toponogov's splitting result can be generalised to Alexandrov geometry. This was first achieved by Milka, with the stronger assumption that an affine function exists \cite{Mil67}, but was later weakened by Burago--Burago--Ivanov to the presence of a line \cite{BBI01}. In the context of Lorentzian length spaces, Beran, Ohanyan, Rott and Solis proved a Splitting Theorem under the presence of global curvature bounds \cite{BORS23}, which we can now restate with the weaker assumption of local curvature bounds.

\begin{thm}[Synthetic Lorentzian Splitting]
Let $(X,d,\ll,\leq,\tau)$ be a connected, globally hyperbolic, regular Lorentzian length space with a proper metric $d$, a time function $T$, and (local) timelike curvature bounded below by zero, which satisfies timelike geodesic prolongation and contains a complete timelike line $\gamma:\R \to X$. 
Then there is a $\tau$- and $\leq$-preserving homeomorphism $f:\R\times S \to X$, where $S$ is a proper, strictly intrinsic metric space of Alexandrov curvature $\geq 0$.
\end{thm}

Observe that the only additional assumption, cf.\ \cite[Theorem 1.4]{BORS23}, made in order to replace global curvature bounds with local ones in the above is the presence of a time function, which is necessary in order to apply Theorem \ref{thm: LorentzianToponogov}. 
Since time functions exist on any second countable, globally hyperbolic, Lorentzian length space (see Proposition \ref{pop:exist-time-function}), this condition is relatively mild.

\subsection{Future work}

The assumption in Theorem \ref{thm: LorentzianToponogov} that the space be a globally hyperbolic \LLS is quite a strong one.
In the metric setting, the assumptions are comparatively mild, e.g. \cite{BGP92} and \cite{LS13} manage to show the theorem for complete length spaces. 
The result can even be shown for non-complete geodesic spaces of curvature bounded below \cite{Pet16}. 
It is therefore only natural to ask whether or not the Toponogov Globalisation Theorem holds in the Lorentzian context under milder assumptions as well. 
Given that \cite{BGP92} globalises curvature bounds using a four-point condition, which was recently adapted to the Lorentzian setting \cite[Definition 4.6]{BKR23}, we are optimistic that the answer is positive and a more general result might be obtained. 
Such a generalisation would also extend the applicability of the Bonnet--Myers theorem, for which the assumptions of the Globalisation Theorem are sufficient but may not all be necessary. 
In particular, the additional assumptions under which the Bonnet--Myers theorem holds for global curvature bounds are weaker than the local version, aside from the bounds themselves.

In the metric case, a powerful consequence of the Toponogov Globalisation Theorem is that the Hausdorff dimension of an Alexandrov space is the same at all neighborhoods in the space \cite{BGP92}. 
A similar notion of dimension has been proposed for \LLSs by McCann and S\"amann \cite{MS22} and it is reasonable to expect that Theorem \ref{thm: LorentzianToponogov} can be used to make an analogous statement.

%bibliography
\bibliographystyle{abbrv}
\addcontentsline{toc}{section}{References}
\bibliography{references} 
\end{document}